\documentclass[12pt]{amsart}
\usepackage{amssymb, amstext, amscd, amsmath}
\usepackage[all]{xy}
\usepackage[all]{xy}
\newtheorem{theorem}{Theorem}[section]

\newtheorem{corollary}[theorem]{Corollary}

\newtheorem{definition}[theorem]{Definition}

\newtheorem{example}[theorem]{Example}

\newtheorem{lemma}[theorem]{Lemma}

\newtheorem{proposition}[theorem]{Proposition}
\newtheorem{remark}[theorem]{Remark}

\makeatletter
\def\@cite#1#2{{\m@th\upshape\bfseries%
[{#1\if@tempswa{\m@th\upshape\mdseries, #2}\fi}]}} \makeatother

\newcommand{\bbC}{{\mathbb{C}}}

\newcommand{\bbN}{{\mathbb{N}}}

\newcommand{\bbT}{{\mathbb{T}}}
\newcommand{\bbZ}{{\mathbb{Z}}}

\newcommand{\A}{{\mathcal{A}}}

\newcommand{\C}{{\mathcal{C}}}

\newcommand{\E}{{\mathcal{E}}}
\newcommand{\F}{{\mathcal{F}}}
\newcommand{\G}{{\mathcal{G}}}
\renewcommand{\H}{{\mathcal{H}}}
\newcommand{\I}{{\mathcal{I}}}
\newcommand{\J}{{\mathcal{J}}}
\newcommand{\K}{{\mathcal{K}}}
\renewcommand{\L}{{\mathcal{L}}}
\newcommand{\M}{{\mathcal{M}}}

\renewcommand{\O}{{\mathcal{O}}}
\renewcommand{\P}{{\mathcal{P}}}

\newcommand{\T}{{\mathcal{T}}}

\newcommand{\X}{{\mathcal{X}}}

\newcommand{\fA}{{\mathfrak{A}}}

\renewcommand{\phi}{\varphi}
\newcommand{\upchi}{{\raise.35ex\hbox{\ensuremath{\chi}}}}
\def\gs{\sigma}
\def\ga{\alpha}
\def\gl{\lambda}

\def\gm{\gamma}
\def\eps{\varepsilon}

\newcommand{\id}{{\operatorname{id}}}

\newcommand{\spn}{\operatorname{span}}

\newcommand\Span{\mathop{\rm span}}

\newcommand{\ca}{\mathrm{C}^*}

\newcommand{\Ginfty}{{\mathcal{G}}^{(\infty)}}

\newcommand{\sca}[1]{\left\langle#1\right\rangle}
\newcommand{\lsca}[1]{\left[#1\right]}
\newcommand\0{\vec{0}}
\newcommand{\vrt}{\G^{(0)}}
\newcommand{\vrtm}{\G^{(0)}_{-}}
\newcommand{\edg}{\G^{(1)}}
\newcommand{\xtau}{X_{\tau}}
\newcommand{\atau}{A_{\tau}}
\newcommand{\ptau}{\phi_{\tau}}

\begin{document}

\title[$\ca$-correspondences and multivariable dynamics]
{Contributions to the theory of $\ca$-correspondences with applications
to multivariable dynamics}

\author[E. Kakariadis]{Evgenios T.A. Kakariadis}
\address{Department of Mathematics\\University of Athens
\\ 15784 Athens \\GREECE}
\email{mavro@math.uoa.gr}

\author[E.G. Katsoulis]{Elias~G.~Katsoulis}
\address{ Department of Mathematics\\University of Athens
\\ 15784 Athens \\GREECE \vspace{-2ex}}
\address{\textit{Alternate address:} Department of Mathematics
\\East Carolina University\\ Greenville, NC 27858\\USA}
\email{katsoulise@ecu.edu}

\begin{abstract}
Motivated by the theory of tensor algebras and multivariable $\ca$-dynamics, we revisit two fundamental techniques in the theory of $\ca$-correspondences, the ``addition of a tail'' to a non-injective $\ca$-correspondence~\cite{MuTom04} and the dilation of an injective $\ca$-correspondence to an essential Hilbert bimodule~\cite{Pim, Sch00}. We provide a very broad scheme for ``adding a tail'' to a non-injective $\ca$-correspondence; our scheme includes the ``tail'' of Muhly and Tomforde as a special case. We illustrate the diversity and necessity of our tails with several examples from the theory of multivariable $\ca$-dynamics.
We also exhibit a transparent picture for the dilation of an injective $\ca$-correspondence to an essential Hilbert bimodule.
As an application of our constructs, we prove two results in the theory of multivariable dynamics that extend results from~\cite{DavR, Pet, Cun}. We also discuss the impact of our results on the description of the $\ca$-envelope of a tensor algebra as the Cuntz-Pimsner algebra of the associated $\ca$-correspondence~\cite{KatsKribs06, MS}.
\end{abstract}

% \thanks{Second author was partially supported by a grant from ECU}

\date{}
\maketitle

\section{Introduction}
Initial motivation for this paper came from a recent result of Davidson and Roydor \cite{DavR}. Specifically, Davidson and Roydor prove that the Cuntz-Pimsner algebra of a multivariable dynamical system for a \textit{commutative} $\ca$-algebra is Morita equivalent to a crossed product of a $\ca$-algebra by an injective endomorphism, thus extending earlier results or claims by Peters~\cite{Pet} and Cuntz~\cite{Cun}. A key step in their proof is a result of a particular interest to us: the Cuntz-Pimsner algebra of a non-injective multivariable system for a commutative $\ca$-algebra is Morita equivalent to the Cuntz-Pimsner algebra of an injective one. The Cuntz-Pimsner algebras of multivariable dynamical systems include as special cases Cuntz's twisted tensor products by $\O_n$ \cite{Cun} and therefore the result of Davidson and Roydor provides an alternative description for these algebras but only in the commutative case. Since several important cases of twisted tensor products by $\O_n$ involve non-commutative $\ca$-algebras, it is desirable to establish a similar result in the non-commutative setting as well. This is accomplished for \textit{arbitrary} $\ca$-algebras in Theorem~\ref{multitail}, Theorem~\ref{injectivemulti} and Corollary~\ref{main} of the present paper. But the paper contains much more than just these generalizations.

Davidson and Roydor prove their results by making heavy use of Gelfand theory. In order to extend their results to the non-commutative setting, we must follow an alternate route that avoids both Gelfand theory and an error in \cite{DavR}.
  This route comes from the theory of $\ca$-correspondences and in particular from a fundamental techniques in that theory: the addition of a tail to a non-injective $\ca$-correspondence (Muhly and Tomforde \cite{MuTom04}. As it appears in the current literature, this technique is insufficiently developed for our purposes. This shifts the main focus of this paper from multivariable dynamics to the abstract theory of $\ca$-correspondences. Consequently, the majority of this paper is occupied with the enrichment and further development of the ``addition of a tail'' to a non-injective $\ca$-correspondence.

  Additional motivation for the further development of this techniques comes from the abstract characterization of the $\ca$-envelope for a tensor algebra of a $\ca$-correspondence. Building on earlier work of Muhly and Solel ~\cite{MS} and Fowler, Muhly and Raeburn~\cite{FMR}, Katsoulis and Kribs~\cite{KatsKribs06} have shown that the $\ca$-envelope of a tensor algebra $\T_X^+$ of a $\ca$-correspondence $(X, A, \phi_X)$ is isomorphic to the Cuntz-Pimsner algebra $\O_X$. In Theorem~\ref{lmain}, by adding a tail to a non-injective correspondence $(X, A, \phi_X)$ we obtain that $\ca_e(\T_{X}^+)$ is a full corner of a Cuntz-Pimsner algebra $\O_Y$ for of an essential Hilbert bimodule $(Y, B, \phi_Y)$. Therefore it becomes important to fully develop and understand the addition of a tail, as it leads to a variety of Morita-equivalent pictures for $\ca_e(\T_{X}^+)$ and therefore have an impact the classification theory for tensor algebras. This result also motivates us to re-examine another important technique in the theory of $\ca$-correspondences: the dilation of an injective $\ca$-correspondence to an essential Hilbert bimodule (Pimsner \cite{Pim}, Schweizer \cite{Sch00}). This is also done in this paper.

  By ``adding a tail'' to a non-injective correspondence we mean that the original correspondence is naturally embedded in an injective correspondence so that the Cuntz-Pimsner algebras associated with both correspondences are (strongly) Morita equivalent. This technique originated in the theory of graph algebras and it was extended to the theory of $\ca$-correspondences by Muhly and Tomforde \cite{MuTom04}. For each non-injective $\ca$-correspondence, Muhly and Tomforde construct a (unique) tail associated with that correspondence. Even though their construction has had significant applications \cite{KatsKribs06}, it is not flexible enough for our needs. To make the case, we show in Proposition~\ref{MTtail} that by adding the Muhly-Tomforde tail to the $\ca$-correspondence of a dynamical system, we obtain a $\ca$-correspondence that lies outside the class of correspondences associated with dynamical systems. In order to bypass this and other limitations, we upgrade the whole construction and we build new types of tails. We can now associate to each $\ca$- correspondence an uncountable family of tails, parameterized by graphs of correspondences satisfying certain compatibility conditions (thus making contact with the recent work of Deaconu et al.~\cite{Deac}). This is done in Theorem~\ref{non-injective case} and, as the reader might suspect, the elaborate part is verifying Morita equivalence. The end result is a very broad class of tails for $\ca$-correspondences, whose variety and usefulness is illustrated by several examples. In particular, in Theorem~\ref{multitail} we present a tail very much different from that of Muhly and Tomforde or Davidson and Roydor, which is used to show that the Cuntz-Pimsner $\ca$-algebra of an arbitrary multivariable system is indeed Morita equivalent to the Cuntz-Pimsner algebra of an injective one, thus resolving one of the motivating questions for this paper.

  By dilating an injective $\ca$-correspondence to an essential Hilbert bimodule, we simply mean that we naturally embed the original correspondence in an essential Hilbert bimodule, so that both correspondences have the same Cuntz-Pimsner $\ca$-algebra associated with them. This technique goes back to the seminal paper of Pimsner~\cite{Pim} but formally it was introduced in a paper by Schweizer~\cite{Sch00} regarding the simplicity of the Cuntz-Pimsner algebras associated with \textit{arbitrary} $\ca$-correspondences. In spite of its importance, this construction, and in particular the left action, is difficult to describe or explain its properties. In Section~\ref{secdilations} we give a very simple and natural picture for the dilation of an injective $\ca$-correspondence to an essential Hilbert bimodule. We emphasize that we do not claim originality here for proving the \textit{existence} of a dilation. What is new here is the picture that we provide.

  In Section \ref{section;mult} we include the applications of our theory to multivariable $\ca$-dynamics. As we mentioned earlier, in Theorem~\ref{multitail} we show that the Cuntz-Pimsner $\ca$-algebra of an arbitrary multivariable system is Morita equivalent to the Cuntz-Pimsner algebra of an injective one. In Theorem~\ref{injectivemulti}, we show that the Cuntz-Pimsner algebra of an \textit{injective} multivariable system is a crossed product $B\times_{\beta} \bbN$ of a $\ca$-algebra $B$ by an injective endomorphism $\beta$. Even though Theorem~\ref{injectivemulti} is proved by abstract means, we have good control on the algebra $B$. For instance, if $A$ is an AF $\ca$-algebra then the same holds for $B$ (and of course, if $A$ is commutative, then $B$ is an AH $\ca$-algebra, as in \cite{DavR}). We remark that the version of a crossed product by an endomorphism that we use in Theorem~\ref{injectivemulti}, even though closely related to that of Paschke~\cite{Pas}, seems to be new. This applies in particular in the case where the $\ca$-algebra $B$ in $B\times_{\beta} \bbN$ is non-unital. Given the examples in this paper, we believe that this crossed product is worthy of further investigation. In general, we anticipate that there will be additional applications for the results of this paper, that go beyond the theory of $\ca$-dynamics. We plan to pursue this in a future paper.

  \section{Preliminaries}

A $\ca$-correspondence $(X,\A,\phi_X)$ consists of a $\ca$-algebra $A$, a Hilbert $A$-module $X$ and a
(perhaps degenerate) $*$-homomorphism $\phi_X\colon A \rightarrow \L(X)$.

A (Toeplitz) representation $(\pi,t)$ of a $\ca$-correspondence into
a $\ca$-algebra $B$, is a pair of a $*$-homomorphism $\pi\colon A \rightarrow B$ and a linear map $t\colon X \rightarrow B$, such that
\begin{enumerate}
 \item $\pi(a)t(\xi)=t(\phi_X(a)(\xi))$,
 \item $t(\xi)^*t(\eta)=\pi(\sca{\xi,\eta}_X)$,
\end{enumerate}
for $a\in A$ and $\xi,\eta\in X$. An easy application of the $\ca$-identity shows that
\begin{itemize}
\item[(3)] $t(\xi)\pi(a)=t(\xi c)$
\end{itemize}
is also valid. A representation $(\pi , t)$ is said to be \textit{injective} iff $\pi$ is injective; in that case $t$ is an isometry.

The $\ca$-algebra generated by a representation $(\pi,t)$ equals the closed linear span of $t^n(\bar{\xi})t^m(\bar{\eta})^*$, where for simplicity $\bar{\xi}\equiv (\xi^{(1)},\dots,\xi^{(n)})\in X^n$ and $t^n(\bar{\xi})\equiv t(\xi_1)\dots t(\xi_n)$.
For any
representation $(\pi,t)$ there exists a $*$-homomorphism
$\psi_t:\K(X)\rightarrow B$, such that $\psi_t(\Theta^X_{\xi,\eta})=
t(\xi)t(\eta)^*$.

Let $J$ be an ideal in $\phi_X^{-1}(\K(X))$; we say that a
representation $(\pi,t)$ is $J$-coisometric if
\[
 \psi_t(\phi_X(a))=\pi(a), \text{ for any } a\in J.
\]
The representations $(\pi,t)$ that are $J_{X}$-coisometric, where
\[
 J_X=\ker\phi_X^\bot \cap \phi_X^{-1}(\K(X)),
\]
are called \emph{covariant representations}~\cite{Kats04}. We define the Cuntz-Pimsner-Toeplitz algebra $\T_X$ as the
universal $\ca$-algebra for all Toeplitz representations of $(X,A,\phi_X)$. Similarly, the Cuntz-Pimsner algebra $\O_X$ is the universal
$\ca$-algebra for all covariant representations of $(X,A,\phi_X)$.

A concrete presentation of both $\T_{X}$ and $\O_{X}$ can be given
in terms of the generalized Fock space $\F_{X}$ which we now
describe. The \emph{Fock space} $\F_{X}$ over the correspondence $X$
is the direct sum of the $X^{\otimes n}$ with the structure of a
direct sum of $\ca$-correspondences over $A$,
\[
\F_{X}= A \oplus X \oplus X^{\otimes 2} \oplus \dots .
\]
Given $\xi \in X$, the (left) creation operator $t_{\infty}(\xi) \in
\L(\F_{X})$ is defined as
\[
t_{\infty}(\xi)(a , \zeta_{1}, \zeta_{2}, \dots ) = (0, \xi a, \xi
\otimes \zeta_1, \xi \otimes \zeta_2, \dots),
\]
where $\zeta_n \in X^{\otimes n}$, $n \geq 0$ and $\zeta_0=a\in A$. (Here $X^{\otimes 0}\equiv A$, $X^{\otimes 1} \equiv X$ and $X^{\otimes n}= X
\otimes X^{\otimes n-1}$, for $n\geq 2$, where $\otimes$ denotes the interior tensor product with
respect to $\phi_X$.) For any $a \in A$,
we define $\pi_{\infty}(a)\in \L(\F_{X})$ to be the diagonal operator
with $\phi_X(a)\otimes id_{n-1}$ at its $X^{\otimes n}$-th entry. It
is easy to verify that $( \pi_{\infty}, t_{\infty})$ is a
representation of $X$ which is called the \emph{Fock representation}
of $X$. Fowler and Raeburn \cite{FR} (resp. Katsura \cite{Kats04})
have shown that the $\ca$-algebra $\ca( \pi_{\infty}, t_{\infty})$
(resp $\ca( \pi_{\infty}, t_{\infty})/ \K(\F_{X})J_{X}$) is
$*$-isomorphic to $\T_{X}$ (resp. $\O_{X}$).

\begin{definition}
The \emph{tensor algebra} $\T_{X}^+$ of a $\ca$-correspondence \break
$(X,A,\phi_X)$ is the norm-closed subalgebra of $\T_X$ generated by
all elements of the form $\pi_{\infty}(a), t_{\infty}^n(\bar{\xi})$, $a \in A$, $\bar{\xi} \in \X^n$, $n \in \bbN$.
\end{definition}

The tensor algebras for $\ca$-correspondences were pioneered by Muhly and Solel in \cite{MS}. They form a broad class of non-selfadjoint operator algebras which includes as special cases Peters' semicrossed products \cite{Pet2}, Popescu's non-commutative disc algebras \cite{Pop3}, the tensor algebras of graphs (introduced in \cite{MS} and further studied in \cite{KaK} and the tensor algebras for multivariable dynamics \cite{DavKat}, to mention but a few.

There is an important connection between $\T_X^+$ and $\O_X$ given in the following
theorem of \cite{KatsKribs06}. Recall that, for an operator algebra
$\fA$ and a completely isometric representation $\iota \colon \fA
\rightarrow A$, where $A=\ca(\iota(\fA))$, the pair $(A,\iota)$ is
called \emph{a $\ca$-cover for $\fA$}. The \emph{$\ca$-envelope of
the operator algebra $\fA$} is the universal $\ca$-cover $(A,\iota)$
such that, if $(B,\iota')$ is any other $\ca$-cover for $\fA$, then
there exists a (unique) $*$-epimorphism $\Phi:B \rightarrow A$, such
that $\Phi(\iota'(a))=\iota(a)$, for any $a\in \fA$ . For the
existence of the $\ca$-envelope see \cite{Ar06, DrMc05, Ham79}.

\begin{theorem} \label{KKenv}
\textup{\cite[Theorem 3.7]{KatsKribs06}.} The $\ca$-envelope of the
tensor algebra $\T^+_X$ is $\O_X$.
\end{theorem}

 One of the fundamental examples in the theory of $\ca$-correspondences  are the $\ca$-algebras of directed graphs, which we now describe. (For more details see~\cite{Raeb}.)

Let $\G$ be a countable directed graph with vertex set $\G^{(0)}$, edge set
$\G^{(1)}$ and range and source maps $r$ and $s$ respectively.
A family of partial isometries $\{ L_e\}_{ e\in
\G^{(1)}}$ and projections $\{ L_p \}_{p\in \G^{(0)}}$ is said to obey
the Cuntz-Krieger relations associated with $\G$ if and only if
they satisfy
\[
(\dagger)  \left\{
\begin{array}{lll}
(1)  & L_p L_p = 0 & \mbox{$\forall\, p,q \in \G^{(0)}$, $ p \neq q$}  \\
(2) & L_{e}^{*}L_f = 0 & \mbox{$\forall\, e, f \in \G^{(1)}$, $e \neq f $}  \\
(3) & L_{e}^{*}L_e = L_{s(e)} & \mbox{$\forall\, e \in \G^{(1)}$}      \\
(4)  & L_e L_{e}^{*} \leq L_{r(e)} & \mbox{$\forall\, e \in \G^{(1)}$} \\
(5)  & \sum_{r(e)=p}\, L_e L_{e}^{*} = L_{p} & \mbox{$\forall\, p
\in \G^{(0)}$ with $|r^{-1}(p)|\neq 0 , \infty$}
\end{array}
\right.
\]
The relations $(\dagger)$ have been refined in a series of papers
by the Australian school and reached the above form in \cite{BHRS,
RS}. All refinements involved condition $(5)$ and as it stands
now, condition $(5)$ gives the equality requirement for
projections $L_{p}$ such that $p$ is not a source and receives
finitely many edges. (Indeed, otherwise condition $(5)$ would not
be a $\ca$-condition.)

It can been shown that there exists a universal $\ca$-algebra,
denoted as $\O_{\G}$, associated with the relations $(\dagger)$.
Indeed, one constructs a single family of partial isometries and
projections obeying $(\dagger)$. Then, $\O_{\G}$ is the
$\ca$-algebra generated by a `maximal' direct sum of such
families. It turns out that there $\O_{\G}$ is the Cuntz-Pimsner algebra of a certain $\ca$-correspondence~\cite{MS}. The associated Cuntz-Pimsner -Toeplitz algebra is the universal algebra for the first four relations in $(\dagger)$ and is denoted as $\T_{\G}$.

%%%%%%%%%%%%%%%%%%%%%%%%%%%%%%%%%%%%%%%%%%%%%%%%%%%%%%%%%%%%
%%%%%%%%%%%%%%%%%%%%%%%%%%%%%%%%%%%%%%%%%%%%%%%%%%%%%%%%%%%%%%%%%%%%%%%%%%%%%%%%%%
%%%%%%%%%%%%%%%%%%%%%%%%%%%%%%%%%%%%%%%%%%%%%%%%%%%%%%%%%%%%%%%%%%%%%%%%%%%%%%%%%%%%%%%%%%%%%%%%%%
%%%%%%%%%%%%%%%%%%%%%%%%%%%%%%%%%%%%%%%%%%%%%%%%%%%%%%%%%%%%%%%%%%%%%%%%%%%%%%%%%%%%%%%%%%%%%%%%%%%%%%%%%%%%%%%%%%%%%%%%%%

\section{A family of tails for $\ca$-correspondences} \label{section non-injective case}

In this section we offer new methods for "adding a tail" \begin{footnote}
{Perhaps "adding a graph of correspondences" would be a more accurate description here.} \end{footnote}
to a non-injective
$\ca$-correspondence. Previously the only such method available in the literature
was that of Muhly and
Tomforde in \cite{MuTom04}. It turns out that their method, when applied to specific examples, produces injective correspondences that may not share the properties of the initial (non-injective) correspondence. For instance, in Proposition~\ref{MTtail} we add the tail of \cite{MuTom04} to a correspondence coming from a $\ca$-dynamical system and we show that the resulting $\ca$-correspondence does not come
from any $\ca$-dynamical system. Additional motivation for studying the technique of "adding a tail" cames from the theory of tensor algebras, as we have explained in the introduction.

Let $\G$ be a connected, directed graph with a distinguished sink $p_0 \in \vrt$ and no sources. We assume that $\G$ is contractible at $p_0$. (See Section~\ref{appendix} for the definition and properties of contractible graphs.) By Theorem~\ref{contractiblethm}, there exists a unique infinite path $w_0 = e_1 e_2 e_3\dots$ ending at $p_0$, i.e., $r(w_0)=p_0$. Let $p_n\equiv s(e_n)$, $n \geq 1$.

Let $(A_p)_{p \in \vrt}$ be a family of $\ca$-algebras parameterized by the vertices of $\G$ so that $A_{p_{0}}=A$. For each $e \in \edg$, we now consider a full, right Hilbert $A_{s(e)}$~-~module $X_e$ and a $*$-homomorphism
\[
\phi_e \colon A_{r(e)} \longrightarrow \L(X_e)
\]
satisfying the following requirements.

For $e \neq e_1$, the homomorphism $\phi_e$ are required to be injective and map onto $\K(X_e)$, i.e., $\phi_e(A_{r(e)})= \K(X_e)$. Therefore, each $X_e$, $e \neq e_1$, is an $A_{r(e)}-A_{s(e)}$-equivalence bimodule, in the language of Rieffel.

For $e = e_1$, we require $\K(X_{e_1}) \subseteq \phi_{e_1} (A)$ and
\begin{equation} \label{inj}
J_{X}\subseteq \ker \phi_{e_1} \subseteq \left( \ker \phi_X \right)^{\perp}.
\end{equation}
In addition, there is also a \textit{linking condition}
\begin{equation} \label{linkin}
\phi_{e_1}^{-1}(\K(X_{e_1})) \subseteq \phi_X^{-1}(\K(X))
\end{equation}
required between the maps $\phi_{X}$ and $\phi_{e_1}$.

Let $T_0= c_0 (\, (A_p)_{p \in \vrtm})$ denote the $c_0$-sum of the family $(A_p)_{p \in \vrtm}$, where $\vrtm \equiv \vrt \backslash \{p_0\}$. Consider the set $c_{00}((X_e)_{e \in \edg}) \subseteq c_0 ((X_e)_{e \in \edg})$, consisting of sequences which are zero everywhere but on a finite number of entries. Equip $c_{00}((X_e)_{e \in \edg}) $ with a $T_0$-valued inner product defined by
\[
\sca{ u, v  }(p)= \sum_{s(e)=p} \, u^*_ev_e, \quad p \in \vrtm,
\]
for any $u, v \in c_{00}((X_e)_{e \in \edg})$. Let $T_1$ be the completion of $c_{00}((X_e)_{e \in \edg}) $ with respect to that inner product. Equip now $T_1$ with a right $T_0$~-~action, so that
\[
(ux)_e = u_ex_{s(e)}, \quad e \in \edg,
\]
for any $x \in T_0$, so that the pair $(T_1, T_0)$ becomes a right $T_0$-Hilbert module. The pair $(T_0, T_1)$ is the \textit{tail} for $(X, A, \phi_X)$.

To the $\ca$-correspondence $(X,A,\phi_X)$ and the data
\[
\tau\equiv \Big(\G, (X_e)_{e \in \edg}, (A_p)_{p\in \vrt}, (\phi_{e})_{e \in \edg} \Big),
\]
we now associate
\begin{align} \label{tau}
\begin{split}
\atau &\equiv A\oplus T_0  \\
\xtau &\equiv X \oplus  T_1.
\end{split}
\end{align}
Using the above, we view $\xtau$ as a $\atau$-Hilbert module with the standard right action and inner product for direct sums of Hilbert modules. We also define a left $\atau$-action
$\ptau:  \atau \rightarrow \L(\xtau)$ on $\xtau$ by setting
\[
\ptau(a, x \, )(\xi, u )= (\phi_X(a)\xi, v  ),
\]
where
\[
v_{e}=
\left\{ \begin{array}{ll}
   \phi_{e_1}(a)(u_{e_1}),& \mbox{if $e = e_1$} \\
   \phi_e (x_{r(e)})u_e, & \mbox{otherwise}
   \end{array}
\right.
\]
for $a \in A$, $\xi \in X$, $x \in T_0$ and $u \in T_1$.

\begin{remark} We make the following two conventions regarding our notation.
\begin{itemize}
\item[(i)] The families $A_{e_n}$, $X_{e_n}$ and $\phi_{e_n}$, $n \geq 0$, will be simply denoted as $A_n$, $X_n$ and $\phi_n$ respectively.

\item[(ii)] Elements of $T_0$ will usually be denoted by such letter symbols as $x, y, z$, while letter symbols like $u,v$ will be reserved for elements of $T_1$. Furthermore, if $e \in \edg$ and $u_e \in X_{e}$, then $u_e\chi_e$ will denote the element of $T_1$ which is equal to $u_e$ on the $e$-entry and $0$ anywhere else. A similar convention holds for the symbol $a\chi_p$, $a \in A_p$, $ p \in \vrtm$.
    \end{itemize}
\end{remark}

\begin{proposition}
With the previous notation, $\ptau:\atau \rightarrow \L(\xtau)$ is a well-defined injective
map.
Moreover, if $\phi$ is non-degenerate, then so
is $\ptau$.
\end{proposition}
\begin{proof}
In order to
show that $\ptau$ is well defined, we need to show that $\ptau(a, x)$ defines a bounded operator
on $X \oplus c_{00}((X_e)_{e \in \edg}) $, for any $a \in A$, $x \in T_0$. However, this follows easily from
\cite[Proposition 1.2]{Lan}.

We now verify that $\ptau$ is injective. If $(a,x) \in \ker \ptau$, then
\[
\ptau(a,x)(\xi, 0) = (\phi_X(a)\xi , 0)=0,
\]for any $\xi \in X$, and so $ a \in \ker \phi$. Hence,
\[
\ptau(a,x)(0 , u_1 \chi_{e_1}) = (0, \phi_1 (a)(u_1)\chi_{e_1}) = 0,
\]
for all $u_1 \in X_1$ and so $a \in \ker \phi_1 \subseteq (\ker \phi_X)^{\perp}$, by condition (\ref{inj}).
This implies that $a = 0$.

Similarly, for any $p \in \vrtm$, let $f \in \edg $ with $r(f) = p$. Then
\[
\ptau(a,x)(0 , u_f \chi_{f}) = (0, \phi_{p}(x_p)(u_f) \chi_{f}) = 0,
\]
for all $u_f \in X_f$. Since $f \neq e_1$, $\phi_{p}$ is injective and so $x_p=0$, for all $p \in \vrtm$, i.e.,
$(a , x)=0$ and so $\ptau$ is injective.

Now, assume that $\phi_X \colon A \rightarrow \L(X)$ is non-degenerate.
Equivalently, if $(a_{\mu})_{\mu\in M}$ is an approximate unit for $A$,
then $(\phi_X(a_{\mu}))_{\mu\in \M}$ is an approximate unit for $\L(X)$.
Let $(x_\gl)_{\gl \in \Lambda}$ be the approximate unit for $T_0$
such that $(x_{p, \gl})_{\gl \in \Lambda}$ is an
approximate unit for $A_p$, for any $p \in \vrtm$; then $((a_{\mu},x_\gl))_{(\mu,\gl)\in
\M \times \Lambda}$ is an approximate unit for $\atau$. Moreover,
\begin{align*}
\lim_{(\mu, \gl)} \ptau(a_{\mu}, x_\gl)(\xi, u)&= \lim_{(\mu, \gl)} \left(\phi_X(a_{\mu})(\xi),\big( \phi_e (x_{r(e), \gl})(u_e) \big)_{e \in \edg} \,\right) \\
&= \left(\xi, (u_e)_{e \in \edg}\right) \\
&= (\xi, u)
\end{align*}
since all $\phi_e$, $e \in \edg$, are non-degenerate. Thus $\ptau(a_{\mu},x_\gl)$ converges strictly to
$1_{\L(Y)}$.
\end{proof}

From now on, we fix an injective covariant representation $(\pi,t)$
of $(\xtau,\atau,\ptau)$ that admits a gauge action, say $\{\beta_z\}_{z\in
\bbT}$. We define $\fA$ to be the closed linear span of
$t^n(\bar{\xi}, 0)t^m(\bar{\eta},0)^*$, $n,m\geq0$. This is exactly the
subalgebra $\ca(\pi|_A,t|_X)$ of $\ca(\pi,t)$; indeed,
\begin{align*}
 \pi(a,0)t(\xi,0) =t(\ptau(a,0)(\xi,0))=t(\phi_X(a)\xi,0),
\end{align*}
for any $a \in A$ and $\xi \in X$.

Given any $\xi, \eta \in X$, we have that
$\psi_t(\Theta_{(\xi,0),(\eta,0)})=t(\xi,0)t(\eta,0)^*$
and so $\psi_t(k)\in \fA$, for any $k\in
\K(X)\subseteq \K(\xtau)$. Furthermore for any $a\in A$ satisfying
$\phi_X (a)=\lim_N \sum_n
\Theta_{w_n,\eta_n}$, we have
\begin{equation} \label{rel1}
\begin{split}
 (\phi_X (a),0) & =\lim_N \sum_{n}
 (\Theta_{w_n,\eta_n},0)=\\
 &= \lim_N \sum_n \Theta_{(w_n,0),(\eta_n,0)},
\end{split}
\end{equation}
and so $(\phi_X(a),0) \in \K(\xtau)$.

\begin{lemma} \label{JJinclusion}
$(\phi_1^{-1}(\K(X_1))+J_{X})\oplus 0
\subseteq J_{\xtau}$.
\end{lemma}
\begin{proof}
 Since $\ptau$ is injective we have that $J_{\xtau}= \ptau^{-1}(\K(\xtau))$.
If $a\in J_X$, then (\ref{inj}) implies that $\ptau(a,0)=(\phi(a), 0)$. Since $\phi(a) \in \K(X)$, we have
from (\ref{rel1}) that $\ptau(a,0)\in \K(\xtau)$ and so $ J_{X}\oplus 0
\subseteq J_{\xtau}$.

Similarly, the linking condition (\ref{linkin}), together with relation (\ref{rel1}) implies that $\phi_1^{-1}(\K(X_1))\oplus 0
\subseteq J_{\xtau}$, and the conclusion follows.
\end{proof}

We now need several technical results.

\begin{lemma} \label{tt*}
Let $(\pi,t)$ be a covariant representation of $(\xtau,\atau,\ptau)$, let $\xi \in X$ and let $u_n, v_n\in X_n$, $n\in \bbN$. Then,
\begin{align*}
(i) \quad &t(0,u_1\chi_{e_1})t^* (0,v_1\chi_{e_1} ) \in \ca(\pi|_A,t|_X) \\
(ii) \quad  &t(0,u_1\chi_{e_1})t^* (0,v_n\chi_{e_n} )=0, \quad &n\geq 2 \\
(iii)\quad &t(0,u_m\chi_{e_m})t^* (0,v_n\chi_{e_n}) = \delta_{m,n} \pi(0, a_{n-1}  \chi_{p_{n-1}}) , \quad &m,n\geq 2\\
(iv) \quad &t(\xi,0)t^* (0,v\chi_{e_n})=0, &n\geq 1
    \end{align*}
where $a_{n-1} \in A_{n-1}$, $n \geq 2$, and $\delta_{m,n}$ denotes the Kronecker delta.
\end{lemma}

\begin{proof}
Let $a_0 \in A$ with $\phi_1(a_0)=\Theta_{u_1, v_1}$. By condition (\ref{linkin}), we have $\phi_X(a_0) \in \K(X)$ and so
\[
\ptau(a_0 , 0) = (\phi_X(a_0) , 0) +\Theta_{(0,u_1\chi_{e_1}) , (0,v_1\chi_{e_1})}
\in \K(\xtau).
\]
Hence, $(a_0,0) \in J_{\xtau}$ and so by covariance
\[
\pi(a_0 , 0) = \psi_{t}((\phi_X(a_0) , 0)) + t(0,u_1\chi_{e_1})t^* (0,v_1\chi_{e_1} )
\]
However, (\ref{rel1}) implies that $\psi_{t}((\phi_X(a_0) , 0))  \in \ca(\pi|_A,t|_X)$ and the proof of (i) is complete.

Parts (ii), (iii) for $m\neq n$, and (iv)  follow from the fact that
\[
\Theta_{(0,u_m\chi_{e_m}) , (0,v_n\chi_{e_n})}= \Theta_{(\xi,0) , (0,v_n\chi_{e_n})}=0.
\]
Finally, if $m=n \geq 1$, and $a_{n-1} \in A_{n-1}$ with $\phi_n(a_{n-1}) = \Theta_{u_n ,v_n}$, then
\[
\ptau( 0, a_{n-1}\chi_{p_{n-1}}) = \Theta_{(0,u_n\chi_{e_n} , (0,v_n\chi_{e_n})}.
\]
Hence, $(0, a_{n-1}\chi_{p_{n-1}} \in J_{\xtau}$ and so part (iii), for $m=n$, follows from the covariance
of $(\pi, t)$.
\end{proof}

Now, let $(x_\gl)_{\gl}$ be an approximate unit for $T_0$ such that
$(x_{p ,\gl})_{\gl}$ is an approximate unit for $A_p$, for all $p \in \vrtm$.
Then for any
$t^n(\bar{\xi},\bar{u})t^m(\bar{\eta},\bar{v})^* \in \O_{\xtau}$, we have that the
limit
\begin{align*}
\lim_{\lambda}\,& \pi(0,x_{\gl})t^n(\bar{\xi},\bar{u})t^m(\bar{\eta},\bar{v})^* = \\
\lim_{\gl}\, & \pi(0,x_{\gl})t(\xi^{(1)},u^{(1)})\dots t(\xi^{(n)},u^{(n)})
t^m(\bar{\eta},\bar{v})^* = \\
\lim_{\gl}\,& t\left(0,(0, \{\phi_e(x_{r(e), \gl})(u_e^{(1)})\}_{e\neq e_1})\right)t(\xi^{(2)},u^{(2)})\dots t(\xi^{(n)},u^{(n)}) t^m(\bar{\eta},\bar{v})^* =\\
& t\left(0,( 0, \{u^{(1)}_e\}_{e\neq e_1})\right)t(\xi^{(2)},u^{(2)})\dots t(\xi^{(n)},u^{(n)}) t^m(\bar{\eta},\bar{v})^*,
\end{align*}
exists. Therefore we obtain a projection $Q$ in the multiplier of $\O_{\xtau}$ defined by $(1-Q)F\equiv \lim_\gl \pi(0,x_\gl) F$, $F \in \O_{\xtau}$, and satisfying
\begin{equation} \label{5rel}
\begin{split}
\left\{
\begin{array}{ll}
\mbox{(i)}& Q\pi(a,x)Q=Q\pi(a,x)=\pi(a,x)Q=\pi(a,0) \\
\mbox{(ii)} &Qt(\xi,u)= t(\xi, u_{e_1}\chi_{e_1}) \\
\mbox{(iii)} & t(\xi,u)Q=t(\xi,0) \\
\mbox{(iv)} & Qt(\xi,0)=t(\xi,0)Q=t(\xi,0) \\
\mbox{(v)} & Q\ca(\pi|_A,t|_X)  Q=\ca(\pi|_A,t|_X),
\end{array}
\right.
\end{split}
\end{equation}
for all $a \in A$, $\xi \in X$, $x \in T_0$ and $u \in T_1$.
In the special case when $A$ is unital and $\phi(1_A)=1_A$ we can
replace $Q$ by $\pi(1_A,0)$ and get the same results.

\begin{lemma} \label{linking}
If $(\pi , t)$ is a covariant representation of $(\xtau, \atau, \ptau)$, then
\[
t(0, u_n\chi_{e_n})t(\xi, v)=t(0, u_n\chi_{e_n})t(0,v_{e_{n+1}}\chi_{e_{n+1}}),
\]
for all $u_n \in X_n$, $\xi \in X$ and $v \in T_1$.
\end{lemma}

\begin{proof}
If $(a_{\lambda})_{\gl \in \Lambda}$ ia an approximate unit for $A_n$, then
\vspace{.1in}
\begin{align*}
t(0, u_n\chi_{e_n})t(\xi, v) &= \lim_{\gl} t(0, u_n \chi_{e_n})\pi(0, a_{\gl}\chi_{p_n}) t(\xi, v) \\
&=\lim_{\gl} t(0, u_n \chi_{e_n})t(\ptau(0, a_{\gl}\chi_{p_n})(\xi, v)) \\
&=\lim_{\gl} t(0, u_n \chi_{e_n}) t(0,\phi_n(c_{\gl})(v_{e_{n+1}})\chi_{e_{n+1}})\\
&=t(0, u_n \chi_{e_n})t(0,v_{e_{n+1}}\chi_{e_{n+1}})
\end{align*}
as desired
\end{proof}

\begin{lemma}
If $(\pi , t)$ is a covariant representation of $(\xtau, \atau, \ptau)$ and
$\bar{\xi} \in X^{m}$, $\bar{u} \in T_1^m$, $m \in \bbN$, then
\[
Qt^{m}(\bar{\xi}, \bar{u})=
\sum_{j}\, t^{m_j}(\bar{\eta_j},0)\prod_{k=1}^{l_j}t(0, u_{k,j}\chi_{e_k}),
\]
for appropriate $m_j, l_j \in \bbN$, $\bar{\eta_j} \in X^{ m_j}$ and $u_{k,j} \in X_k$.
\end{lemma}

\begin{proof}The proof follows by induction on $m$. If $m=1$ then by part~(ii) of (\ref{5rel}) we have
\begin{align*}
Qt(\xi , u)&= t(\xi, u_{e_1}\chi_{e_1}) \\
           &=t(\xi, 0) + t(0,  u_{e_1}\chi_{e_1})
\end{align*}
as desired.

Assume now that the result is true for all integers smaller or equal to $m$. Once again, by part (ii) of (\ref{5rel}) we have
\begin{align*}
Qt(\xi',u')t^{m}(\bar{\xi}, \bar{u}) &=Qt(\xi' , 0)t^{m}(\bar{\xi}, \bar{u}) +Qt(0,u'_{e_1}\chi_{e_1})t^{m}(\bar{\xi}, \bar{u})
\end{align*}
However by part (iv) of (\ref{5rel}) and the inductive hypothesis,
\begin{align*}
Qt(\xi' , 0)t^{m}(\bar{\xi}, \bar{u}) &=Qt(\xi' , 0)Qt^{m}(\bar{\xi},\bar{u}) \\
     &=Qt(\xi', 0)\sum_{j}\, t^{m_j}(\bar{\eta_j},0)\prod_{k=1}^{l_j}t(0, u_{k,j}\chi_{e_k})\\
     &=\sum_{j}\, t(\xi', 0)t^{m_j}(\bar{\eta_j},0)\prod_{k=1}^{l_j}t(0, u_{k,j}\chi_{e_k})
     \end{align*}
On the other hand, Lemma \ref{linking} shows that $Qt(0,u'_{e_1}\chi_{e_1})t^{m}(\bar{\xi}, \bar{u})$ also has the desired form and the conclusion follows.
\end{proof}

\begin{lemma}
If $(\pi , t)$ is a covariant representation of $(\xtau, \atau, \ptau)$, then
\[
Qt^{m}(\bar{\xi}, \bar{u})t^{n}(\bar{\eta}, \bar{v})^*Q \in  \ca(\pi|_A,t|_X),
\]
for all
$\bar{\xi} \in X^{m}$, $\bar{\eta} \in X^n$, $\bar{u} \in T_1^m$ and $\bar{v} \in T_!^n$, $m, n \in \bbN$.
\end{lemma}

\begin{proof} In light of the previous Lemma, it is enough to show that any product of the form
\[
t(\xi, 0)\prod_{k=1}^{l_1}t(0, u_{k}\chi_{e_k})\left(t(\eta, 0)\prod_{k=1}^{l_2}t(0, v_{k}\chi_{e_k})\right)^*
\]
belongs to $\ca(\pi|_A,t|_X)$.
If the above product is not 0, then Lemma \ref{tt*} implies that $l_1=l_2\equiv l$. If $l=1$ then the conclusion follows from Lemma \ref{tt*} (i).
If $l\geq 2$, then Lemma \ref{tt*} (iii) shows that
\[
t(0, u_l \chi_{e_l})t(0, v_l \chi_{e_l})^* =\pi(0, a_{l-1} \chi_{p_{l-1}})
\]
and so
\begin{align*}
t(0, u_{l-1} \chi_{e_{l-1}})&t(0, u_l \chi_{e_l})t(0, v_l \chi_{e_l})^*t(0, v_{l-1} \chi_{e_{l-1}})^*  \\ &=t(0, u_{l-1} \chi_{e_{l-1}})\pi(0, a_{l-1}  \chi_{p_{l-1}})t(0, v_{l-1} \chi_{e_{l-1}})^*
\\ &= t(0, u_{l-1}a_{l-1} \chi_{e_{l-1}})t(0, v_{l-1} \chi_{e_{l-1}})^*.
\end{align*}
Continuing in this fashion, the conclusion follows.
\end{proof}

\begin{corollary} \label{corner}
If $(\pi , t)$ is a covariant representation of $(\xtau, \atau, \ptau)$,
then $Q\ca (\pi, t)Q = \ca(\pi|_A,t|_X)$.
\end{corollary}

The previous corollary shows that $\ca(\pi|_A,t|_X)$ is a corner of $\ca (\pi, t)$. The next Lemma will be used in the proof of the main result, Theorem \ref{non-injective case}, to show that this corner is actually full.

\begin{lemma} \label{full}
Let $(\pi , t)$ be a covariant representation of $(\xtau, \atau, \ptau)$ and let $\I$ be an ideal of $\ca (\pi , t)$.
\begin{itemize}
\item[$($i$)$] If $e \in\edg \backslash \{e_1\}$ and either $\pi(0, A_{r(e)}\chi_{r(e)}) \subseteq \I$ or $\pi(0, A_{s(e)}\chi_{s(e)}) \subseteq \I$, then
    \[
    t(0, X_e \chi_{e})\subseteq \I,
    \]
\item[(ii)] If $e \in \edg$ and $t(0 , X_e\chi_e) \subseteq \I$, then $\pi(0, A_{s(e)}) \subseteq \I$.
\item[(iii)]If $p \in \vrtm$ and
\[
t(0, X_e\chi_e)\subseteq \I, \mbox{ for all } e \in r^{-1}(p),
\]
then,
\[
\pi(0, A_p\chi_{p}) \subseteq \I
\]
\end{itemize}
\end{lemma}

\begin{proof} (i). Let $ u_e \in X_e$. Notice that if $(a_{e, \gl})_{\gl}$ is an approximate unit for $A_{s(e)}$, then
\begin{align*}
t(0,u_e\chi_e)&= \lim_{\gl} \, t(0,u_ea_{e, \gl}\chi_e) \\
                &= \lim_{\gl} \, t(0,u_e) \pi(0, a_{e, \gl}\chi_{s(e)})
\end{align*}
and the conclusion follows from the fact that $ \pi(0, a_{e, \gl}\chi_{s(e)}) \in \I$.

The proof of the other inclusion follows from the fact that $\phi_e$ is non-degenerate.

\noindent (ii). If $u_e, u_e' \in X_e$, then,
\[
t(0, u_e \chi_e)^*t(0, u_e' \chi_{e}^{\phantom{*}}) = \pi(0, \langle u_e^{\phantom{'}}, u_e'\rangle \chi_{s(e)})
\]
and the conclusion follows from the fact that $X_e$ is full.

\noindent(iii).  Let $a_p \in A_p$. For any $\xi \in X$ and $u \in T_1$, we have
\[
\ptau((0,a_p\chi_p))(\xi, u)=\big(0,\sum_{e \in r^{-1}(p)} \phi_e(a_p)(u_e)\, \chi_e \big).
\]
Since $\phi_e(a_p) \in \K(X_e)$, for all $ e \in r^{-1}(p)$, we obtain $(0, a_p \chi_p) \in J_{\xtau}$. By covariance,
\begin{equation} \label{span}
\pi(0,a_p \chi_p ) \in \overline{\spn} \{ t(0, u_e \chi_e)t(0, u_e' \chi_{e}^{\phantom{*}})^*\mid u_e^{\phantom{*}},u_e' \in X_e, e \in r^{-1}(p) \}.
\end{equation}
However, our assumptions imply that elements of the form $t(0, u_e \chi_e)$, $e \in r^{-1}(p)$, belong to $\I$. Hence, by (\ref{span}), we have $\pi(0,a_p \chi_p ) \in \I$, as desired.
\end{proof}
We have arrived to the main result of this section.

\begin{theorem}\label{non-injective case}
Let $(X,A,\phi_X)$ be a non-injective C*-correspondence and let $(\xtau, \atau, \ptau)$
 be the $\ca$-correspondence associated with the data
 \[
\tau = \Big(\G, (X_e)_{e \in \edg}, (A_p)_{p\in \vrt}, (\phi_{e})_{e \in \edg} \Big),
 \]
 as defined at the beginning of the section. Then $(\xtau,\atau,\ptau)$ is an injective $\ca$-correspondence and the Cuntz-Pimsner algebra $\O_X$ is a full corner of $\O_{\xtau}$.
\end{theorem}

\begin{proof}
Let $(\pi,t)$ be the universal covariant representation of $(\xtau,\atau,\ptau)$.
It is easy to see that $(\pi|_A,t|_X)$ is an injective
representation of $(X,A,\phi)$ that inherits a gauge action from $(\pi,t)$.
Furthermore, $(\pi|_A,t|_X)$ is a covariant
representation of $(X,A,\phi)$.

Indeed, let $a\in J_X=\ker\phi_X^\bot\cap \phi_X^{-1}(\K(X))$ with $\phi_X(a)= \lim_N
\sum_{n} \Theta_{w_n,\eta_n}$.
Since $\ptau(a,0)=(\phi_X(a),0)$, we have
\[
\ptau(a,0)= \lim_N \sum_{n}
\Theta_{(w_n,0),(\eta_n,0)}.
\]
 Therefore,
\begin{align*}
 \psi_{t|_X}(\phi_X(a))& =\psi_{t|_X}(\lim_N \sum_{n}
 \Theta_{(w_n,0),(\eta_n,0)})
 = \psi_t(\lim_N \sum_{n}
 \Theta_{(w_n,0),(\eta_n,0)})=\\
 &= \psi_t(\ptau(a,0))=\pi(a,0)=\pi|_A(a).
\end{align*}
and the covariance of $(\pi|_A,t|_X)$ follows. Hence, by gauge invariance,
$\ca (\pi|_A,t|_X)$ is isomorphic to $\O_{\xtau}$.

We have seen in Corollary \ref{corner} that $Q\ca (\pi, t)Q = \ca(\pi|_A,t|_X)$, i.e., $\O_{X}$ is a corner of $\O_{\xtau}$. We now show $\O_{X}$ is a \textit{full} corner.
Assume that $\I$ is an ideal of $\ca (\pi, t)$ containing $\ca (\pi|_A,t|_X)$. We are to show that $\I = \ca (\pi, t)$. We start by showing
\begin{equation} \label{laststraw}
t(0, X_{e_1}\chi_{e_1})\subseteq \I.
\end{equation}
Let $u_1 \in \X_1$. Since $\phi_{1}$ is non-degenerate, there exist $a \in \A$ and $u_1' \in X_1$ so that $\phi_1(a_1)(u_1')=u_1$. Hence,
\begin{align*}
\I \ni \pi(a_1,0)t(0, u_1'\chi_{e_1})&= t(0,\phi_1(a_1)(u_1')\chi_{e_1})\\
&=t(0, u_1 \chi_{e_1})
\end{align*}
and the validity of (\ref{laststraw}) follows.

Repeated applications of Lemma~\ref{full}
show now that
\[
\pi(0, A_p \chi_{p}) \subseteq \I
\]
for all $p \in \bigcup_{n \in \bbN} \, S_n$, where $S_n $ are as in Section~\ref{appendix}. However, Theorem~\ref{contractiblethm} shows that $\bigcup_{n \in \bbN} \, S_n=\vrt$ and the conclusion follows.
\end{proof}

Lets see now how the work of Muhly and Tomforde fits in our theory.

\begin{example}[The Muhly-Tomforde tail~\cite{MuTom04}] \label{MTexample}
\textup{Given a (non-injective) correspondence $(X, A, \phi_X)$, Muhly and Tomforde construct in \cite{MuTom04} the tail that results from the previous construction, with respect to data
\[
\tau =\Big(\G, (X_e)_{e \in \edg}, (A_p)_{p\in \vrt}, (\phi_{e})_{e \in \edg} \Big)
\]
defined as follows.
The graph $\G$ is illustrated in the figure below.}

\vspace{.2in}
\[
\xymatrix{&{\bullet^{p_0}} &{\,\, \bullet}^{p_1} \ar[l]^{e_1}&{\,\,\bullet^{p_2}} \ar[l]^{e_2} &{\bullet^{p_3}} \ar[l]^{e_3}&{\,\,\bullet} \ar[l]&\dots \ar[l]}
\]
\vspace{.2in}

\noindent \textup{$A_p=X_e=ker \phi_X$, for all $p \in \vrtm$ and $e \in \edg$. Finally,
\[
\phi_e(a)u_e=au_e,
\]
for all $e \in \edg$, $u_e \in X_e$ and $a \in A_{r(e)}$.}
\end{example}

The tail of Muhly and Tomforde has had significant applications in the theory of $\ca$-correspondences, including a characterization for the $\ca$-envelope of the tensor algebra of a non-injective correspondence~\cite{KatsKribs06}. However, it also has its limitations, as we are about to see.

Let $(X_{\ga}, A)$ be the $\ca$-correspondence canonically associated with a $\ca$-dynamical system $(A, \ga)$, i.e., $X_{\ga}=A$ and $\phi_{X_{\ga}}(a_1)(a_2)=\ga(a_1)a_2$, for $a_1,a_2 \in A$; let $\O_{(A,\ga)}$ be the associated Cuntz-Pimsner $\ca$-algebra. If $\ga$ is not injective, then by using the Muhly-Tomforde tail we obtain an injective $\ca$-correspondence $(Y,B, \phi_Y)$ so that $\O_{(A,\ga)}$ is a full corner of $\O_Y$. Remarkably, $(Y, B, \phi_Y)$ may not come from any $\ca$-dynamical system, as the following result shows.

\begin{proposition} \label{MTtail}
Let $(A, \ga)$ be a $\ca$-dynamical system with $1\in A$ and $\ga(1)=1$. Assume that $\ker \ga \subseteq A$ is an essential ideal of $A$. Let $(X_{\ga}, A)$ the $\ca$-correspondence canonically associated with $(A, \ga)$ as described above. If $(Y, B, \phi_Y)$ denotes the $\ca$-correspondence resulting by adding the Muhly-Tomforde tail to $(X_{\ga}, A)$, then there exists no $*$-homomorphism $\beta$ on $B$ so that
\begin{equation} \label{impos}
\phi_Y(b)(b')=\beta(b)b'
\end{equation}
for all $b, b' \in B$
\end{proposition}

\begin{proof}
For the correspondence $(X_{\ga}, \ga)$ described above, the Muhly-Tomforde tail produces an injective correspondence $(Y, B, \phi_Y)$ with
\begin{align*}
Y&= A\oplus c_{0}(\ker \ga) \\
B&=A\oplus c_{0}(\ker \ga)
\end{align*}
and $\phi_Y$ defined by
\[
\phi_Y\big( a, (c_i)_i\big)\big(a',(c_i')_i\big)
= \big(\ga(a)a',\ga(a)c_1', c_1c_2', c_2c_3',\dots\big),
\]
where $a,a' \in A$ and $(c_i)_i, (c_i')_i \in c_0(\ker \ga)$.

If there was a $*$-homomorphism $\beta$ satisfying (\ref{impos}), then by equating second coordinates in the the equation
\[
\phi_Y\big( 1, (c_i)_i\big)\big(a',(c_i')_i\big)=
\beta \big( 1, (c_i)_i\big)\big(a',(c_i')_i\big)
\]
we would obtain,
\[
c_1'=
\beta \big(1,(c_i)_i \big)_2 c_1',
\]
for all $c_1' \in \ker \ga$. Since $\ker \ga$ is an essential ideal, we have $\ker \ga \ni \beta \big(1,(c_i)_i \big)_2 =1$, a contradiction.
\end{proof}

Therefore, the Muhly-Tomforde tail produces an injective correspondence but not necessarily an injective dynamical system. Nevertheless, there exists a tail that can be added to $(X_{\ga}, A)$ and produce an injective correspondence that comes from a $\ca$-dynamical system. This is done in Example \ref{revisited}.

We finish this section with a discussion regarding Theorem~\ref{non-injective case} and the conditions imposed on the graph $\G$ and the maps $(\phi_e)_{e \in \edg}$.

We ask that the graph $\G$ be contractible. We cannot weaken this assumption to include more general graphs. Indeed, we want the tail associated with the data
 \[
\tau = \Big(\G, (X_e)_{e \in \edg}, (A_p)_{p\in \vrt}, (\phi_{e})_{e \in \edg} \Big),
 \]
to work with any possible Cuntz-Pimsner algebra $\O_X$ that can be ``added on''. This should apply in particular to the Cuntz-Krieger algebra $\O_{G_{p_0}}$ of the (trivial) graph $\G_{p_0}$ consisting only of one vertex $p_0$. By taking $\tau$ to be the ``usual'' tail associated with $\G$, i.e., $X_e=A_e=\bbC L_{p_0}$ and $\phi_e$ left multiplication for all $e$, we see that $\O_{G_{p_0}}$ is a full corner of $O_{\xtau}$ if and only if $\G$ is contractible at $p_0$.

Conditions~(\ref{inj}) and (\ref{linkin}) are also necessary, as the following result suggests.

\begin{proposition}
Let $(X,A,\phi_X)$ be a non-injective C*-correspondence and let
$(\xtau, \atau, \ptau)$ be the $\ca$-correspondence associated with
the data
 \[
\tau = \Big(\G, (X_e)_{e \in \edg}, (A_p)_{p\in \vrt}, (\phi_{e})_{e
\in \edg} \Big),
 \]
as defined at the beginning of the section. If $X_\tau$ is
injective,
\[
(\phi_1^{-1}(\K(X_1))+J_{X})\oplus 0 \subseteq
J_{\xtau},
\]
and the covariant representations of $X_\tau$ restrict
to covariant representations of $\phi_X$, then
\begin{equation*}
 J_{X}\subseteq \ker \phi_{1}
\subseteq \left( \ker \phi_X \right)^{\perp},
\end{equation*}
and the \textit{linking condition}
\begin{equation*}
\phi_{1}^{-1}(\K(X_{1})) \subseteq \phi_X^{-1}(\K(X))
\end{equation*}
holds.
\end{proposition}
\begin{proof}
If $\ptau$ is injective, then every $\phi_e$ is injective for $e\neq
e_1$. For $a\in \ker\phi_{e_1}$ and $c\in \ker\phi_X$, we have that
\[
\ptau(ac,0)(\xi,u)=(\phi_X(ac)\xi, \phi_{1}(ac)(u_{e_1}),0)=0,
\]
so $ac\in \ker\ptau=(0)$. Thus $ \ker\phi_1 \subseteq \left( \ker
\phi_X \right)^{\perp}$.

Let $a \in \phi_{1}^{-1}(\K(X_{1}))$. If $(c_\gl)_\gl$ is an
approximate identity for $\K(X_{1})$, then the proof of Lemma \ref{JJinclusion} implies that,
\begin{align*}
\ptau(a,0)= (\phi_X(a),0) + \lim_\gl \Theta_{(0, \theta(a)\chi_{e_1}),
(0,c_\gl\chi_{e_1})}.
\end{align*}
Since $\ptau(a,0) \in \K(\xtau)$, we have that $(\phi_X(a),0)$ is an adjointable operator in $\L(X) \cap \K(\xtau)$. Thus $\phi_X(a) \in \K(X)$.

Finally, the assumption that the covariant representations of
$X_\tau$ restrict to covariant representations of $\phi_X$, along
with the hypothesis that $J_X \subseteq J_{\xtau}$ is equivalent to
$\phi_\tau(a,0)=(\phi_X(a),0)$, for all $a\in J_X$. This implies $J_{X}\subseteq \ker \phi_{1}$ and we are done.
\end{proof}

%%%%%%%%%%%%%%%%%%%%%%%%%%%%%%%%%%%%%%%%%%%%%%
\section{Cuntz-Pimsner algebras for multivariable $\ca$-dynamics} \label{section;mult}

We now apply the theory of the previous section to multivariable
 $\ca$-dynamics. Apart from their own merit, these applications will also address the necessity of using more elaborate tails than that of Muhly and Tomforde in the process of adding tails to $\ca$-correspondences. This necessity has been already noted in Proposition~\ref{MTtail}.

A \textit{multivariable $\ca$-dynamical system}
is a pair $( A, \ga)$ consisting of
a $\ca$-algebra $A$ along with
a tuple $\ga=( \ga_1, \ga_2 , \dots, \ga_n)$, $n \in \bbN$, of $*$-endomorphisms of $A$. The dynamical system is called injective iff
$\cap_{i=1}^n \, \ker\ga_i =\{0\}$.

 To the multivariable system $( A, \ga)$ we associate a $\ca$-correspondence $(X_{\ga}, A, \phi_{\ga})$ as follows. Let $X_{\ga}=A^n = \oplus_{i=1}^n A$ be the usual right $A$-module. That is
\begin{enumerate}
\item $(a_1,\dots,a_n)\cdot a= (a_1 a ,\dots,a_n a)$,
\item $\sca{(a_1,\dots,a_n), (b_1,\dots,b_n)}=\sum_{i=1}^n
\sca{a_i,b_i}=\sum_{i=1}^n a_i^*b_i$.
\end{enumerate}
Also, by defining the $*$-homomorphism
\begin{align*}
\phi_{\ga}\colon A \longrightarrow \L(X_{\ga})\colon a \longmapsto \oplus_{i=1}^n \ga_i(a),
\end{align*}
$X$ becomes a $\ca$-correspondence over $A$, with $\ker\phi_{\ga}=
\cap_{i=1}^n \ker \ga_i$ and $\phi(A) \subseteq \K(X_{\ga})$. It is
easy to check that in the case where $A$ and all $\ga_i$ are unital, $X$ is finitely generated as an
 $A$-module by the elements
\[
e_1:=(1,0,\dots,0),e_2:=(0,1,\dots,0),\dots,e_n:=(0,0,\dots,1),
\]
where $1\equiv 1_A$. In that case, $(\pi,t)$ is a representation of this
$\ca$-correspondence if, and only if, the $t(\xi_i)$'s are
isometries with pairwise orthogonal ranges and
\[
\pi(c)t(\xi)=t(\xi)\pi(\ga_i(c)), \quad i=1,\dots, n.
\]

\begin{definition}
The Cuntz-Pimsner algebra $\O_{(A, \ga)}$ of a multivariable $\ca$-dynamical system
$(A,\ga)$ is the Cuntz-Pimsner algebra of the
$\ca$-correspondence $(X_{\ga}, A, \phi_{\ga})$ constructed as above
\end{definition}

In the $\ca$-algebra literature, the algebras $\O_{(A, \ga)}$ are denoted as \break $A \times_{\ga} \O_n$ and go by the name "twisted tensor products by $\O_n$". They were first introduced and studied by Cuntz~\cite{Cun} in 1981. In the non-selfadjoint literature, there algebras are much more recent. In \cite{DavKat} Davidson and the second named author introduced the tensor algebra $\T_{(A, \ga)}$ for a  multivariable dynamical system $(A,\ga)$. It turns out that $\T_{(A,\ga)}$ is completely isometrically isomorphic to the tensor algebra for the $\ca$-correspondence $(X_{\ga}, A, \phi_{\ga})$. As such, $\O_{(A, \ga)}$ is the $\ca$-envelope of $\T_{(A, \ga)}$. Therefore, $\O_{(A, \ga)}$ provides a very important invariant for the study of isomorphisms between the tensor algebras $\T_{(A, \ga)}$.

We now apply the construction of Section \ref{section non-injective case}
to the
$\ca$-correspondence defined above. The graph $\G$ that we associate with $(X_{\ga}, A, \phi_{\ga})$ has no loop edges and a single sink $p_0$. All vertices in $\vrt \backslash \{p_0\}$ emit $n$ edges, i.e., as many as the maps involved in the multivariable system, and receive exactly one. In the case where $n=2$, the following figure illustrates $\G$.
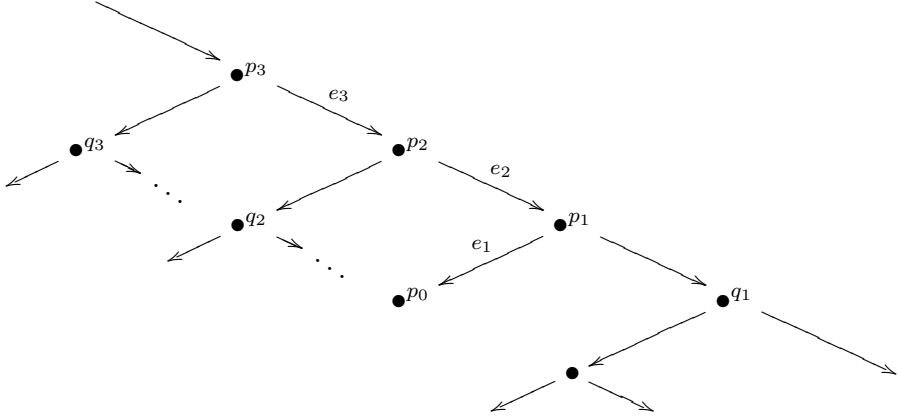
\vspace{.2in}
\[
\xygraph{
!{<0cm,0cm>;<2.15cm,0cm>:<0cm,1cm>::}
!{(1,3)}*+{\empty}="a"
!{(2,2)}*+{\bullet^{p_3}}="b"
!{(1,1)}*+{\bullet^{q_3}}="c"
!{(3,1)}*+{\bullet^{p_2}}="cc"
!{(0.5,.5)}*+{\empty}="e"
!{(1.5,.5)}*+{\ddots}="f"
!{(4,0)}*+{\bullet^{p_1}}="j"
!{(5,-1)}*+{\bullet^{q_1}}="q"
!{(6,-2)}*+{\empty}="u"
!{(7,-3)}*+{\empty}="bb"
!{(3,-1)}*+{\bullet^{p_0}}="p"
!{(4,-2)}*+{\bullet}="w"
!{(3.5,-2.5)}*+{\empty}="v"
!{(4.5,-2.5)}*+{\empty}="y"
!{(2,0)}*+{\bullet^{q_2}}="i"
!{(1.5,-.5)}*+{\empty}="k"
!{(2.5,-.5)}*+{\ddots}="l"
"b" :"cc"^{e_3}
"cc":"j"^{e_2}
"j":"q"
"j":"p"_{e_1}
"a":"b"
"q":"w"
"q":"u"
"cc":"i"
"b":"c"
"c":"e" "c":"f"
"i":"k"
"i":"l"
"w":"v"
"w":"y"
}
\]
\vspace{.2in}
Clearly, $\G$ is $p_0$-accessible. (In the case $n=2$, $S_0 =\{p_0\}$,
\[S_1=\{  p_0, p_1, p_2, \dots\}\cup \{q_1, q_2, \dots \}
\]
and so on.) There is also a unique infinite path $w$ ending at $p_0$ and so the requirements of Theorem~\ref{contractiblethm} are satisfied, i.e., $\G$ is contractible at $p_0$.

Let $\J \equiv \cap_{i=1}^{n} \, \ker \alpha_i$ and let $M(\J)$ be the multiplier algebra of $\J$. Let $\theta \colon A \longrightarrow M(\J)$ the map that extends the natural inclusion $\J \subseteq M(\J))$. Let $X_e= A_{s(e)} =\theta(A)$, for all $e \in \edg$, and consider $(X_e, A_{s(e)})$ with the natural structure that makes it into a right Hilbert module.

For $e \in \edg \backslash \{e_1\}$ we define $\phi_{e}(a)$ as left multiplication by $a$. With that left action, clearly $X_e$ becomes an $A_{r(e)}-A_{s(e)}$-equivalence bimodule. For $e = e_1$, it is easy to see that
\[
\phi_{e_1}(a)(\theta(b))\equiv \theta(ab), \quad a,b \in A
\]
defines a left action on $X_{e_1}=\theta(A)$, which satisfies both (\ref{inj}) and (\ref{linkin}).

 For the $\ca$-correspondence $(X_{\ga}, A, \phi_{\ga})$ and the data
  \[
 \tau = \Big(\G, (X_e)_{e \in \edg}, (A_p)_{p\in \vrt}, (\phi_{e})_{e \in \edg} \Big),
   \]
   we now let $((X_{\ga})_{\tau} , A_{\tau}, (\phi_{\ga})_{\tau}) $ be the $\ca$-correspondence constructed as in the previous section. For notational simplicity $((X_{\ga})_{\tau} , A_{\tau}, (\phi_{\ga})_{\tau}) $ will be denoted as $(\xtau, \atau, \ptau)$. Therefore
\begin{align*}
\atau &= A\oplus c_0 (\vrtm, \theta(A))  \\
\xtau &= A^n \oplus  c_0 (\G^{(1)} ,\theta(A)).
\end{align*}
Now label the n-edges of $\G$ emitting from each $p \in \vrtm$ as $p^{(1)}, p^{(2)}, \dots,p^{(n)}$, with the convention that that if $e_{n}, p_n$, $n\geq0$, are as in Section \ref{section non-injective case}, then $e_n$ is labeled as $p_n^{(1)}$. It is easy to see now that the mapping
\[
c_0 (\G^{(1)},\theta(A)) \ni u \longmapsto \oplus_{i=1}^n \{ u(p^{(i)})\}_{p \in \vrt} \in \oplus_{i=1}^n \, c_0 (\vrtm,\theta(A))
\]
establishes a unitary equivalence
\begin{align*}
\xtau &=  A^n \oplus  c_0 (\G^{(1)} ,\theta(A)) \\
        &\cong  \oplus_{i=1}^n \, \left( A\oplus c_0 (\vrtm,\theta(A)\right)
\end{align*}
between the Hilbert $A$-module $\xtau$ and the $n$-fold direct sum of the $\ca$-algebra
$ A\oplus c_0 (\vrtm,\theta(A))$, equipped with the usual $A \oplus c_0 (\vrtm,\theta(A))$-right action and inner product.

It only remains to show that the left action on $\xtau$ comes from an $n$-tuple of $*$-endomorphisms of
$A\oplus c_0 (\vrtm,\theta(A))$. This is established as follows.

For any $i=1, 2, \dots, n$ and  $(a, x) \in A\oplus c_0 (\vrtm,\theta(A))$ we define
\begin{equation*}
\hat{\ga}_i (a, x)=( \ga_i (a), \gamma_i(a,x))
\end{equation*}
where $\gamma_i(a,x) \in c_0 (\vrtm,\theta(A))$ with
\[
\gamma_i(a,x)(p)=
\left\{
\begin{array}{ll}
\theta(a), &\mbox{if } p^{(i)}=e_0, \\
x(r(p^{(i)})), &\mbox{otherwise.}
\end{array}
\right.
\]
It is easy to see now that $\left( A\oplus c_0 (\vrtm,\theta(A)), \hat{\ga}_1, \dots,\hat{\ga}_n \right)$ is a multivariable dynamical system, so that the $\ca$-correspondence associated with it is unitarily equivalent to
$(\xtau, \atau, \ptau)$.

We have therefore proved

\begin{theorem} \label{multitail}
If $(A, \ga)$ is a non-injective multivariable $\ca$-dynamical system, then there exists an injective multivariable $\ca$-dynamical system $(B, \beta)$ so that the associated Cuntz-Pimsner algebras $\O_{(A, \ga)}$ is a full corner of
$\O_{(B, \beta)}$. Moreover, if $A$ belongs to a class
$\, \C$ of $\, \ca$-algebras which is invariant under quotients and $c_0$-sums, then $B\in \C$ as well. Furthermore, if $(A, \ga)$ is non-degenerate, then so is $(B, \beta)$.
\end{theorem}

\begin{example} \textup{(The Muhly-Tomforde tail, revisited). Theorem~\ref{multitail} shows how to correct the situation in Example~\ref{MTexample} in order to avoid the pathology of Proposition~\ref{MTtail}.}

  \textup{If $(X_{\ga}, A)$ is the $\ca$-correspondence canonically associated with a $\ca$-dynamical system $(A, \ga)$, i.e., $X_{\ga}=A$ and $\phi_{X_{\ga}}(a_1)(a_2)=\ga(a_1)a_2$, for $a_1,a_2 \in A$, then the appropriate tail for $(X_{\ga}, A)$ comes from the data
  \[
\tau =\Big(\G, (X_e)_{e \in \edg}, (A_p)_{p\in \vrt}, (\phi_{e})_{e \in \edg} \Big)
\]
where $\G$ is as in Example~\ref{MTexample} \textit{but} for any $p \in \vrtm$ and $e \in \edg$,
\[A_p=X_e= \theta(A),
\]
where $\theta \colon A \longrightarrow M(\ker \ga)$ is the map that extends the natural inclusion $\ker \ga \subseteq M(\ker \ga))$ in the multiplier algebra. Finally
\[
\phi_e(a)u_e=au_e,
\]
for all $e \in \edg$, $u_e \in X_e$ and $a \in A_{r(e)}$.}
\end{example}

The reader familiar with the work of Davidson and Roydor may have noticed that the the arguments in the proof of Theorem~\ref{multitail}, when applied to multivariable systems over commutative $\ca$-algebras produce a tail which is different from that of Davidson and Roydor in \cite[Theorem 4.1]{DavR}. It turns out that the proof of \cite[Theorem 4.1]{DavR} contains an error and the technique of Davidson and Roydor does not produce a full corner, as claimed in \cite{DavR}. We illustrate this by examining their arguments in the following simple case.

\begin{example} \label{revisited} \textup{ (The Davidson-Roydor tail~\cite{DavR}). Let $\X\equiv\{u ,v \}$ and consider the maps $\sigma_i \colon  \X \rightarrow \X$, $i=1,2$, with $\sigma_i(u) =v$ and $\sigma_i(v)=v$. Set $\sigma \equiv (\sigma_1 , \sigma_2)$ and let $\O_{(X, \sigma)}$ be the Cuntz-Pimsner algebra associated with the multivariable system $(\X, \sigma)$, which by \cite{DavKat} is the $\ca$-envelope of the associate tensor algebra.}

\textup{We now follow the arguments of \cite{DavR}. In order to obtain $\O_{(\X, \sigma)}$ as a full corner of an injective Cuntz-Pimsner algebra, Davidson and Roydor add a tail to the multivariable system. They define
\[
T= \{(u,k)\mid k<0\} \mbox{ and }\X^{T}= \X\cup T.
\]
For each $1\leq i \leq2$, they extend $\sigma_i$ to a map $\gs_i^T\colon \X^T \rightarrow \X^T$ by
\[\gs^T(u,k)= (u, k+1) \mbox{ for $k<-1$ and  } \sigma_i^T(u,-1)=u.
\]
They then consider the new multivariable system $(X^T, \sigma^T)$ and its associated Cuntz-Pimsner algebra $\O{(X^T,\gs^T)}$.}

\textup{It is easy to see that the Cuntz-Pimsner algebra $\O_{(\X, \sigma)}$ for the multivariable system $(\X, \sigma)$ is the Cuntz-Krieger algebra $\O_{\G}$ of the graph $\G$ illustrated below,}

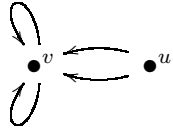
\vspace{.2in}
\[
\xymatrix{
& {\bullet^v} \ar@(u,ul)[] \ar@(d,dl)[] &{\, \bullet^u \,} \ar@/^/[l] \ar@/_/[l] }
\]
\vspace{.2in}

 \noindent \textup{while the Cuntz-Pimsner algebra $\O{(X^T,\gs^T)}$ is isomorphic to the Cuntz-Krieger algebra $\O_{\G^T}$ of the following graph $\G^{T}$, where for simplicity we write $u_k$ instead of $(u,k)$, $k<0$,}

\vspace{.2in}
\[
\xymatrix{
& {\bullet^v} \ar@(u,ul)[] \ar@(d,dl)[] &{\, \bullet^{u } \,} \ar@/^/[l] \ar@/_/[l] &{\bullet^{u_{-1}}\,}\ar@/^/[l]^f \ar@/_/[l]_e  &{\bullet^{u_{-2}}\,}\ar@/^/[l] \ar@/_/[l] &{\dots}\ar@/^/[l] \ar@/_/[l]}
\]
\vspace{.2in}

\textup{In \cite[page 344]{DavR}, it is claimed that the projection P associated with the characteristic function of $\X \subseteq \X^T$ satisfies $P\O_{(\X^T, \sigma^T)}P= \O_{(\X, \sigma)}$ and so $\O_{(\X, \sigma)}$ is a corner of $\O_{(\X^T, \sigma^T)}$.
In our setting, this claim translates as follows: if $P=L_{u}+L_{v}$, then $P\O_{\G^{T}}P=\O_{\G}$. However this is not true. For instance, $P(L_f L_e^*)P=L_f L_e^* \notin \O_{\G}$.}
\end{example}

We now describe the Cuntz-Pimsner algebra of an injective multivariable system as a crossed product of a $\ca$-algebra $B$ by an endomorphism $\beta$. We begin with the pertinent definitions.

\begin{definition} \label{covariant}
Let $B$ be a (not necessary unital) $\ca$-algebra and let $\beta$ be an injective endomorphism of $B$. A \textit{covariant} representation $(\pi, v)$ of the dynamical system $(B, \beta)$ consists of a (perhaps degenerate) $*$-representation $\pi$ of $B$ and an isometry $v$ satisfying
\begin{itemize}
\item[(i)] $\pi(\beta(b))=v\pi(b)v^*, \forall \, b \in B$, i.e., $v$ \textit{implements} $\beta$,
\item[(ii)] $v^*\pi(B)v\subseteq \pi(B)$, i.e., v is \textit{normalizing} for $\pi(B)$,
\item[(iii)] $v^k(v^*)^k \pi(B) \subseteq \pi(B), \forall \, k \in \bbN.$
    \end{itemize}
    \end{definition}

The crossed product $B \times_{\beta} \bbN$ is the universal $\ca$-algebra associated with this concept of a covariant representation for $(B, \beta)$. Specifically, $B \times_{\beta} \bbN$ is generated by $B$ and $VB$, where $V$ is an isometry satisfying (i), (ii) and (iii) in Definition \ref{covariant} with $\pi=\id$.
Furthermore, for any covariant representation $(\pi, v)$ of $(B, \beta)$, there exists a $*$-homomorphism $\hat{\pi}\colon B \times_{\beta} \bbN \rightarrow B(\H)$
extending $\pi$ and satisfying $\hat{\pi}(bV) = \pi(b)v$, for all $b \in B$.

In the case where $B$ is unital and $\pi$ non-degenerate, condition (iii) is redundant and this version of a crossed product by an endomorphism was introduced by Paschke~\cite{Pas}; in the generality presented here, it is new. It has the advantage \footnote{One here needs to observe that conditions (i), (ii) and (iii) in Definition~\ref{covariant} imply that $\hat{\pi}\left(B \times_{\beta} \bbN \right)$ is generated by polynomials of the form $\pi(b_0) + \sum_{k}\, \pi(b_k ) v^k +\sum_{l} \,(v^*)^l \pi( b_l')$,  \, $b_0 , b_k , b_l \in B. $  } that for any covariant representation of $(\pi, v)$ of $(B, \beta)$ admitting a gauge action, the fixed point algebra of $(\pi, v)$ equals $\pi(B)$. This allows us to claim a gauge invariance uniqueness theorem for $B \times_{\beta} \bbN$: if $(\pi, v)$ is a faithful covariant representation of $(B, \beta)$ admitting a gauge action, then the $\ca$-algebra generated by $\pi(B)$ and $\pi(B)v$ is isomorphic to $B \times_{\beta} \bbN$.

There is a related concept of a crossed product by an endomorphism which we now discuss. For a $\ca$-algebra $B$ and an injective endomorphism $\beta$, Stacey \cite{Stac} imposes on a covariant representation $(\pi, v)$ of $(B, \beta)$ only condition (i) from Definition~\ref{covariant}. He then defines the crossed product $B \rtimes_{\beta} \bbN$ to be the universal $\ca$-algebra associated with his concept of a covariant representation for $(B, \beta)$. Muhly and Solel have shown~\cite{MS2} that in the case where $B$ is unital, Stacey's crossed product is the Cuntz-Pimsner algebra of a certain correspondence. Using a gauge invariance uniqueness theorem one can prove that if the isometry V in $B \rtimes_{\beta} \bbN$ satisfies condition (ii) in Definition~\ref{covariant}, then $B \rtimes_{\beta} \bbN \simeq  B \times_{\beta} \bbN$.

In the case where $A$ is a commutative $\ca$ algebra, the following result was proven in \cite{DavR} by using Gelfand theory to construct a new multivariable dynamical system. It turns out that the concept of a faithful covariant representation suffices to prove the result for arbitrary $\ca$-algebras.

\begin{theorem} \label{injectivemulti}
If $(A, \ga)$ is an injective multivariable system, then there exists a $\ca$-algebra $B$ and an injective endomorphism $\beta$ of $B$ so that $\O_{(A, \ga)}$ is isomorphic to the crossed product algebra $B\times_{\beta}\bbN$. Furthermore, if $A$ belongs to a class $\C$  which is invariant under direct limits and tensoring by $M_k(\bbC)$, $k\in \bbN$, then $B$ also belongs to $\C$.
\end{theorem}

\begin{proof}
Consider the unitizations
\begin{align*}
A' &=A +\bbC 1 \supseteq  A               \\
X'&=\oplus_{i=1}^{n}\, A' \supseteq  X
\end{align*}
and
\[
\ga_i'\colon \A_1 \longrightarrow \A_1 ;\, a+\gl1\longmapsto \ga_i(a)+ \gl 1.
\]
(Here we understand $A'$ to be the unique unital $\ca$-algebra that contains $A$ as and ideal and has the property that $A' \slash A \simeq \bbC$.)

Let $(\pi, t)$ be a faithful covariant representation of $(X_{\ga '}, A', \phi_{\ga ' })$. Notice that the restriction of $(\pi, t)$ on $(X_{\ga}, A, \phi_\ga)$ is a (faithful) covariant representation admitting a gauge action and so the $\ca$-algebra generated by it is $\O_{(A, \ga)}$.

Let $e_i= (0, \dots, 1, \dots, 0)$, where the $1$ appears in the $i$th- position, and consider the subalgebras
\begin{align*}
B_m&\equiv \overline{\Span} \left\{ t^m(\bar{\xi})t^m(\bar{\eta})^* \mid \bar{\xi}. \bar{\eta}\in X_{\ga}^m \right\}  \\
&=\overline{\Span} \left\{ \left(\prod _{k=1}^{m} t(e_{i_{k}})\right) \pi(a)\left(\prod _{l=1}^{m} t(e_{i_{l}})\right)^* \mid \, a \in A  \right\} \\
&\simeq M_{n^m}(A).
\end{align*}
Since $\phi_{\ga'}$ acts on $X_{\ga'}$ by rank-one operators, the covariance of $(\pi, t)$ implies that
\[
\pi(A)\subseteq B_1\subseteq \dots \subseteq B_{k} \subseteq B_{k+1} \subseteq \dots
\]
Let
\[
B\equiv \overline{\, \, \, \,  \bigcup_{m=1}^{\infty}\, B_m }
\]
and set
\[
v= \frac{1}{\sqrt{n}} \sum_{i=1}^{n} \, t(e_i).
\]
 The isometry $v$ leaves invariant $B$ and therefore it defines an injective endomorphism
\[
\beta \colon B \longrightarrow B \colon  x \longrightarrow v x v^*.
\]
Let $\ca(B, Bv)$ be the $\ca$-algebra generated by $B $ and $Bv$. Then $\ca(B, Bv)$ inherits from $(\pi, t)$ its natural gauge action that leaves invariant $B$ and twists $v$. Since $v$ satisfies (ii) and (iii) in Definition~\ref{covariant}, we have by gauge invariance
\[
\ca(B, Bv) \simeq B \times_{\beta} \bbN.
 \]
On the other hand, for $a \in A$ and each $i=1,2,\dots ,n $, we have
\[
t(e_i)\pi(a)=\sqrt{n} t(e_i)\pi(a)t(e_i)^* v\in B v
\]
since $t(e_i)\pi(a)t(e_i)^* \in B_2$. But elements of the form
\[
\phantom{XXXX}t(e_i) \pi(a)= t(0, \dots, 0,a,0,\dots,0),  \quad a\in A, \, i=1,2, \dots,n
\]
span $t(X_\ga)$ and so
\[
\ca(B, Bv) = \O_{(A, \ga)}.
\]
Hence $B\times_{\beta}\bbN  \simeq \O_{(A, \ga)} $, as desired.
\end{proof}

Even in the case where $A$ is unital but $\ga$ is degenerate, we still need to pass to the unitization $A'$ in the above proof in order to ensure that $B$ is a directed limit of matrix algebras over $A$. Notice however that $A$ may not contain the unit operator $1$ of the Hilbert space on which $(\pi , t)$ acts. Nevertheless, when both $A$ and $\ga$ are unital, we do not need to pass to the unitization. In that case $B$ contains the unit operator $1$ and so by the earlier discussion on Stacey's crossed product, we have $\O_{(A,\ga)}\simeq B \rtimes_{\beta}\bbN$.

Combining Theorem~\ref{injectivemulti} with Theorem~\ref{multitail} we obtain

\begin{corollary} \label{main}
If $(A, \ga)$ is a multivariable system, then there exists a $\ca$-algebra $B$ and an injective endomorphism $\beta$ of $B$ so that $\O_{(A, \ga)}$ is isomorphic to a full corner of the crossed product algebra $B\times_{\beta}\bbN$. Furthermore, if $A$ belongs to a class $\C$  which is invariant under direct limits, quotients and tensoring by $M_k(\bbC)$, $k\in \bbN$, then $B$ also belongs to $\C$.
\end{corollary}

Finally, let us give a quick application of Theorem~\ref{injectivemulti}, that readily follows  from Paschke's result \cite{Pas} on the simplicity of $B \times_{\beta}\bbN$.

\begin{corollary}
Let $A$ be a UHF $\ca$-algebra and let $\ga=(\ga_1,\dots,\ga_n)$ be a multivariable system with $n\geq2$. If $\ga_i(1)=1$, for all $i=1,2,\dots,n$, then $\O_{(A,\ga)}$ is simple.
\end{corollary}

%%%%%%%%%%%%%%%%%%%%%%%%%%%%%%%%%%%%%%%%%%%%%%%%%%%%%%%%%%%%%%%%%%%%%%%%%%%%%%%%

\section{The $\ca$-envelope of a tensor algebra}\label{applications}

Motivation for the study in this paper comes from the fact that the Cuntz-Pimsner algebra $\O_X$ of a $\ca$-correspondence $(X, A, \phi_X)$ is the $\ca$-envelope of the associated tensor algebra $\T_X^+$, as we remarked in Theorem~\ref{KKenv}. We may paraphrase Theorem~\ref{KKenv} as follows.

\begin{theorem}\label{lmain}
If $(X, A, \phi_X)$ is a $\ca$-correspondence over $A$, then the $\ca$-envelope
of the tensor algebra $\T_{X}^+$ can identified as a full corner of
a Cuntz-Pimsner algebra $\O_Y$ for of an essential Hilbert bimodule $(Y, B, \phi_Y)$. Moreover, if
$(X, A, \phi_X)$ is injective, then $\ca_e(\T_{X}^+)$ is $*$-isomorphis to $\O_Y$.
\end{theorem}

\begin{proof} If $(X, A, \phi_X)$ is injective, then this follows by Theorem
\ref{injective case}. If $(X, A, \phi_X)$ is not injective, then by choosing any
"tail" in Theorem \ref{non-injective case} (there is at
least one such tail, as shown by Muhly and Tomforde \cite{MuTom04}),  $\ca_e(\T_{X}^+)=\O_X$ is a
full corner of a Cuntz-Pimsner algebra of an injective
$\ca$-correspondence, and the conclusion follows
\end{proof}

Theorem \ref{main} raises a host of relevant questions and directions for further investigation. By adding various tails to a non-injective $\ca$-correspondence $(X, A, \phi_X)$, we obtain a family of Morita-equivalent pictures for $\ca_e(\T_{X}^+)$. For instance, we can do this with the tensor algebras for non-injective multivariable systems. Can this be used in the classification program for such algebras? In particular, we wonder if there is a familiar description for the Cuntz-Pimsner algebra $\O_Y$ of Proposition~\ref{MTtail}, that results from adding the Muhly-Tomforde tail to the $\ca$-correspondence of a non-injective dynamical system $(A,\ga)$.

%%%%%%%%%%%%%%%%%%%%%%%%%%%%%%%%%%%%%%%%%%%%%%%%%%%%%%%%%%%%%%%%%%%%%%%%%%%%%
%%%%%%%%%%%%%%%%%%%%%%%%%%%%%%%%%%%%%%%%%%%%%%%%%%%%%%%%%%%%%%%%%%%%%%%%%%%%%%%%%%%%%%%%%%%%%%%%%%

\section{Appendix: Dilations of $\ca$-correspondences} \label{secdilations}

The concept of dilating an injective $\ca$-correspondence to an essential Hilbert bimodule was formally introduced  by
Schweizer \cite{Sch00} with the purpose of studying the simplicity of Cuntz-Pimsner algebras. However, most of the elements of his construction already appear in the seminal paper of Pimsner \cite{Pim}.
In this section we give an overview of their theory with the purpose of establishing notation and providing a very transparent picture for the dilation of an injective correspondence to an essential Hilbert bimodule. Therefore we do not claim originality for the statements of the results included here; only for their proofs and the picture that they provide.

In the $\ca$-literature, the term ``Hilbert bimodule'' has several interpretations. For us, it will be used as follows.

\begin{definition}
\textup{ A  \emph{Hilbert $A$-bimodule} $(X,A,\phi_X)$ is a
$\ca$-correspondence $(X,A,\phi_X)$ together with a left inner
product $\lsca{\cdot,\cdot}\colon X \times X \rightarrow A$, which
satisfy:
\begin{enumerate}
\item $\lsca{\phi_X(a)\xi,\eta}= a\lsca{\xi,\eta}$,
$\lsca{\xi,\eta}=\lsca{\eta,\xi}^*$, $\lsca{\xi,\xi}\geq 0$,
\item $\phi_X(\lsca{\xi,\eta})\zeta=\xi\sca{\xi,\zeta}$
\end{enumerate}
for $\xi, \eta, \zeta \in X$, and $a\in A$.}
\end{definition}
The last equation implies that
$\phi_X(\lsca{\xi,\eta})=\Theta^X_{\xi,\eta}$. It is clear that
Hilbert bimodules are a special case of $\ca$-correspondences. Let
$I_X$ be the ideal,
\[
 I_X=\overline{\Span}\{\lsca{\xi,\eta}: \xi,\eta \in X\},
\]
in $A$. Using the very definitions, one can prove that $a\in
\ker\phi_X$, if and only if $a\in I_X^\bot$. Hence, $\phi_X$ is
$*$-injective, if and only if, the Hilbert $A$-bimodule $X$ is
\emph{essential}, i.e. when the ideal $I_X$ is essential in $A$. The
ideal $I_X$ is associated to the covariant representations of the
$\ca$-correspondence $(X,A,\phi_X)$ in the following fundamental
way.

\begin{lemma}\textup{(\cite{Kats03})}
If a Hilbert $A$-bimodule $(X,A,\phi_X)$ is considered as a
$\ca$-correspondence over $A$, then $J_X=I_X$.
\end{lemma}

Hence, for $\xi,\eta\in X$, the element $\lsca{\xi,\eta}\in A$ is
identified with the unique element $a\in J_X$ such that
$\phi_X(a)=\Theta^X_{\xi,\eta}$. This is the big advantage (and a characterizing property \cite{Kats04}) of Hilbert bimodules, i.e., $J_X$ is $*$-isomorphic to $\K(X)$.
\begin{definition}
A correspondence $(Y, B, \phi_Y)$ is said to be a dilation of the correspondence $(X, A, \phi_X)$ iff
$(X, A, \phi_X)$ sits naturally inside $(Y, B, \phi_Y)$, i.e., $(X, A, \phi_X)$ is unitarily equivalent to a sub-correspondence of $(Y, B, \phi_Y)$, and the associated Cuntz-Pimsner algebras $O_X$ and $O_Y$ are $*$-isomorphic.
\end{definition}

In order to describe the dilation of the correspondence $(X, A, \phi_X)$ to an essential Hilbert bimodule, consider the
directed system
\[
A \stackrel{\rho_0}{\longrightarrow} \L(X)
\stackrel{\rho_1}{\longrightarrow}  \L(X^{\otimes 2}) \stackrel{\rho_2} \longrightarrow \cdots,
\]
where
\begin{align*}
&\rho_0=\phi_X \colon A=\L(A)\longrightarrow \L(X),\\
&\rho_n \colon \L(X^{\otimes n}) \longrightarrow \L(X^{\otimes n+1})\colon r
\longmapsto r\otimes \id_1 , \, n\geq 1,
\end{align*}
and let $B\equiv \varinjlim (\L(X^{\otimes n}), \rho_n)$.

Consider also the directed system of Banach spaces
\[
\L(A,X) \stackrel{\sigma_0}{\longrightarrow}
\L(X,X^{\otimes 2}) \stackrel{\sigma_1}{\longrightarrow} \cdots ,
\]
where
\begin{align*}
 \sigma_n \colon \L(X^{\otimes n}, X^{\otimes n+1}) \rightarrow
\L(X^{\otimes n+1}, X^{\otimes n+2}) \colon  s \mapsto s\otimes \id_1 , \, n\geq 1,
\end{align*}
and let $Y \equiv \varinjlim (\L(X^{\otimes n}, X^{\otimes n+1}), \gs_n)$. Note that the map
\[
\partial \colon X \longrightarrow
\L(A,X) \colon \xi \longmapsto \partial_{\xi},
\]
where $\partial_{\xi}(a)=\xi a$, $\xi \in X$, maps a copy of $X$ isometrically into \break $\K(A,X) \subseteq Y$.

If $r\in \L(X^{\otimes n})$, $s\in \L(X^{\otimes n},X^{\otimes
n+1})$ and $[r], [s]$ their equivalence classes in $B$ and $Y$ respectively,
then we define $[s]\cdot [r]:= [sr]$. From this, it is easy to define a right $B$-action on $Y$.
Similarly, we may define a $B$-valued right inner product on $Y$ by setting
\begin{align*}
\sca{[s'],[s]}_Y \equiv [(s')^*s] \in B.
\end{align*}
for $s, s' \in \L(X^{\otimes n},X^{\otimes n+1})$, $n \in \bbN$,
and then extending to $Y \times Y$. Finally we define a $*$-homomorphism $\phi_Y \colon B
\rightarrow \L(Y)$ by setting
\[
\phantom{XXX} \phi_Y([r])([s]) \equiv [rs], \quad r\in \L(X^{\otimes n}), s\in
\L(X^{\otimes n-1},X^{\otimes n}), n\geq0
\]
and extending to all of $B$ by continuity. We therefore have a left $B$-action on $Y$ and
thus $(Y, B, \phi_Y)$ becomes a $\ca$-correspondence.

The following diagrams depicts the above construction in a heuristic way: the right action is "defined" through the diagram
\begin{align*}
\xymatrix{ B \ar@{-->}[d]_{\cdot} & A \ar[r]^{\phi_X\equiv
\rho_0} \ar@{-->}[d]_{\cdot} & \L(X) \ar[r]^{\rho_2}
\ar@{-->}[d]_{\cdot} & \L(X^{\otimes 2}) \ar[r]^{\rho_3}
\ar@{-->}[d]_{\cdot} &
\dots \ar@{-->}[d]_{\cdot} \\
Y & \L(A,X) \ar[r]^{\sigma_0}  & \L(X,X^{\otimes 2})
\ar[r]^{\sigma_1} & \L(X^{\otimes 2},X^{\otimes 3}) \ar[r]^{\phantom{XXX} \sigma_3} &  \dots
}
\end{align*}
while the left action is "defined" through the diagram
\begin{align*}
\xymatrix{ B \ar@{-->}[d]_{\phi_{Y}} & A \ar[r]^{\phi_X\equiv \rho_0}
 & \L(X) \ar[r]^{\rho_2} \ar@{-->}[dl]_{\id} &
\L(X^{\otimes 2}) \ar[r]^{\rho_3} \ar@{-->}[dl]_{\id} &
\dots \ar@{-->}[dl]_{\id} \\
Y & \L(A,X) \ar[r]^{\sigma_0}  & \L(X,X^{\otimes 2})
\ar[r]^{\sigma_1} & \L(X^{\otimes 2},X^{\otimes 3}) \ar[r]^{} &
\dots }
\end{align*}

\begin{proposition}
Let $(X, A, \phi_X)$ be a faithful
$\ca$-correspondence and let $(Y, B, \phi_Y)$ be as above. Then $(Y, B, \phi_Y)$ can be equipped with a left $B$-valued inner product so that it becomes an essential Hilbert bimodule.
\end{proposition}
\begin{proof}
It is easy to see that the mapping
\begin{align*}
\lsca{[s],[r]} \equiv  [sr^*] \in B,
\end{align*}
where $s,r \in \L (X^{\otimes n }, X^{\otimes n+1})$, defines a left inner product in $Y \times Y$, such that $\xi
\sca{\eta , \zeta} = \lsca{\xi, \eta} \zeta$, for any $\xi, \eta,
\zeta \in Y$. Thus $(Y, B, \phi_Y)$ becomes a Hilbert $B$-bimodule.

In order to show that $(Y, B, \phi_Y)$ is faithful, it suffices to
prove that $\phi_{Y}$ is injective on every
$\L(X^{\otimes n})$, $n \geq 0$. Let $\phi_n$ denote the restriction of
$\phi_{Y}$ on $\L(X^{\otimes n})$ and $\partial_m:=\partial \otimes \id_{m} \in
\L(X^{\otimes m},X^{\otimes m+1})$, for $m \geq 0$.

For $n=0$, let
$a\in \ker\phi_{ 0}$. Therefore $a\in A$ and
\[
0=\phi_0(a)(\partial_{\xi})=\partial_{\phi_X(a)\xi},
\]
for any $\xi \in X$. Thus $a\in \ker\phi_X$, i.e., the class of $a$
in $B$ is the zero class.

For an arbitrary $n\geq 1$, let $r\in
\ker\phi_n$. Then $\phi_n(r)(\partial_n(\xi))=0$, for any $\xi \in
X$, and so
\begin{align*}
\phi_n(r)(\partial_n(\xi))(\eta_1\otimes \cdots \eta_{n-1})&=
r\circ \partial_n(\xi)(\eta_1 \otimes \cdots \eta_{n-1})\\
&= r( \xi \otimes \eta_1 \otimes \cdots \eta_n),
\end{align*}
for any $\xi, \eta_1,\dots, \eta_{n-1} \in X$. Thus, $r=0$ in that case as well and $\phi_Y$ is injective.
\end{proof}

The bimodule $(Y, B, \phi_Y)$ is too big for a dilation of $(X, A, \phi_X)$
and so we need to restrict to a smaller bimodule in order to obtain isomorphism of the associated Cuntz-Pimsner algebras.

Let $A_{\infty} \subseteq B$ be the $\ca$-algebra that is generated by all the copies
of $\K(X^{\otimes n})$, $n\geq 0$, inside $B$
and let $X_{\infty} \subseteq Y$ be the closed subspace generated by all
copies of $\K(X^{\otimes n},X^{\otimes n+1})$, $n\geq 0$, in $Y$.
We now show that triple $(X_{\infty}, A_{\infty}, \phi_{\infty})$ is the right dilation for $(X, A, \phi_X)$, where $\phi_{\infty}$ denotes the restriction of $\phi_Y$ on $\A_{\infty}$. First we need a lemma.

\begin{lemma}
Let $(X, A, \phi_X)$ be a faithful
$\ca$-correspondence and let $(X_{\infty}, A_{\infty}, \phi_{\infty})$ be as above. Then the Hilbert bimodule
$(X_{\infty}, A_{\infty}, \phi_{\infty})$ is essential.
\end{lemma}

\begin{proof}
We will show that
$J_{\infty}=\overline{\lsca{X_{\infty},X_{\infty}}}$ is essential in $A_{\infty}$; let $c\in J_{\infty}^\bot$. Since
\[
A_{\infty}=
\overline{\Span}\{[a], [r] \mid\, a\in A, r\in \K(X^{\otimes n}), n \geq
1\}
\]
and
\[
J_{\infty} =
\overline{\Span}\{[r]\mid \,r \in \K(X^{\otimes
n}), n\geq 1\},
\]
Corollary 1.5.8 in \cite{Ped79} implies that $A_{\infty}= A+J_{\infty}$.
Hence there exist $a\in A$ and a $d \in J_{\infty}$ such
that
\begin{equation} \label{spliteq}
c=[a] - d.
\end{equation}
If $(u_\gl)_\gl $ be an approximate identity in $\K(X)$, then
\begin{align*}
\begin{split}
0=c[u_\gl] &= a[u_\gl]-d[u_\gl] \\
           &=[\phi_X(a)u_\gl] -d[u_\gl].
           \end{split}
           \end{align*}
Since $([u_\gl ])_\gl \subseteq J_\infty$ is an approximate unit,
$\big([\phi_X(a)u_\gl]\big)_\gl$, and so by injectivity $\big(\phi_X(a)u_\gl\big)_\gl$, is a Cauchy sequence, converging to some $k \in \K(X)$. On the other hand, $\lim_{\gl}\,u_\gl\xi=\xi$, for all $\xi \in \X$. Hence $\phi_X(a)=k\in \K(X)$ and so $[a] \in \J_\infty$. Therefore (\ref{spliteq}) implies that $c \in J_\infty \cap J_{\infty}^\bot =\{0\}$ and we are done.
\end{proof}

\begin{theorem}\label{injective case}
Let $(X, A, \phi_X)$ be an injective $\ca$-correspondence and let $(X_{\infty}, A_{\infty}, \phi_{\infty})$ be as above. Then $(X_{\infty}, A_{\infty}, \phi_{\infty})$ is an essential Hilbert bimodule and its Cuntz-Pimsner algebra $\O_{X_\infty}$ is isomorphic to $\O_X$.
\end{theorem}

\begin{proof} Since $\phi_X$ is injective, we have that $J_X \subseteq
\phi_X^{-1}(\K(X))$, hence $J_X$ is contained in $J_\infty$. Thus if
$(\pi,t)$ is a an injective covariant representation that admits a
gauge action of $(X_\infty,A_{\infty},\phi_{X_{\infty}})$, then $(\pi|_A,t|_X)$ is also an
injective covariant representation of $(X,A,\phi_X)$ that admits a
gauge action. So $\O_X$ embeds in $\O_{X_{\infty}}$. It suffices
then to prove that $\O_X$ is exactly $\O_{X_{\infty}}$. But this
is true since the fixed point algebra of $\O_X$ is $\pi(A_{\infty})$ and $t(X_{\infty})=
\overline{t(X)\pi(A_{\infty})}$.
\end{proof}

%%%%%%%%%%%%%%%%%%%%%%%%%%%%%%%%%%%%%%%%%%%%%%%%%%%%%%%%%%
%%%%%%%%%%%%%%%%%%%%%%%%%%%%%%%%%%%%%%%%%%%%%%%%%%%%%%%%%%%%%%%%%%%%%%
\section{Appendix: Contractible graphs} \label{appendix}

In this section we verify the properties of contractible graphs needed in
Section~\ref{section non-injective case}.

\begin{definition} Let $\G$ be a connected, directed graph $\G$ with a distinguished vertex $p_0 \in \vrt$ and no sources. The graph $\G$ is said to be \textit{contractible at $p_0$} if the subalgebra $\bbC L_{p_0} \subseteq \O_{\G}$ is a full corner of the Cuntz-Krieger algebra $\O_{\G}$.
\end{definition}

Therefore the $\ca$-algebra of a contractible graph is Morita equivalent to the compact operators in a very strong sense. In this section we give an algorithmic way of describing contractible graphs which will allow us to complete the proof of Theorem~\ref{non-injective case}.

If $\G$ is a directed graph with no sources and $S \subseteq \vrt$, then we define
\begin{align*}
\P_{s}(S) &=S\cup \{ p \in \vrt \mid \exists w \in \Ginfty \mbox{ with } r(w) \in S, s(w)=p \}\\
\E(S) &= \{ e \in \edg \mid s(e) \in \P_{s}(S) \} \\
\P_{r}(S)&= \{ p \in \vrt \mid r^{-1}(p) \subseteq \E(S) \},
\end{align*}
where $\Ginfty$ denotes the collection of all finite paths in $\G$.

Given any connected, directed graph $\G$ with a distinguished vertex $p_0 \in \vrt$ and no sources, we set
\begin{align*}
S_0 &= \{ p_0 \}  \\
S_n &= \P_{r}(S_{n-1}), \mbox{ for }n\geq1.
\end{align*}

\begin{definition}
Let $\G$ be a connected. directed graph with a distinguished vertex $p_0 \in \vrt$ and no sources. A vertex $p \in \vrt$ is said to be \textit{$p_0$-accessible} (or simply, accessible) if
\[
p \in \bigcup_{n \in \bbN} \, S_n= \bigcup_{n \in \bbN} \, \P_{r}(S_n)
\]
Similarly, an edge $e \in \edg$ is said to be $p_0$-accessible iff $e \in \bigcup_{n \in \bbN}\, \E(S_n)$. The graph $\G$ is said to be $p_0$-accessible iff every element of $\vrt$ is accessible.
\end{definition}

We can state now the desired characterization of contractible graphs.

\begin{theorem} \label{contractiblethm}
Let $\G$ be a connected directed graph with a distinguished vertex $p_0 \in \vrt$ and no sources. Then $\bbC L_{p_0} \subseteq \O_{\G}$ is a full corner iff the following two conditions are satisfied:
\begin{itemize}
\item[$(i)$] The graph $\G$ is $p_0$-accessible.
\item[$(ii)$] There exists exactly one infinite path $w$ with $r(w)=p_0$.
\end{itemize}
\end{theorem}

\begin{proof}
It is easy to see that if $\G$ satisfies conditions (i) and (ii), then
$L_{p_0} \subseteq \O_{\G}$ is a full corner.

Assume now that $\bbC L_{p_0}$ is a full corner of $\O_{\G}$. We start by verifying the following two claims.

\vspace{.1in}
\noindent \textbf{Claim 1}. \textit{The graph $\G$ is acyclic.}
\vspace{.05in}

\noindent \textit{Proof of Claim 1.} By way of contradiction, assume that $\G$ has a cycle $w$. Since $\O_{\G}\otimes \K(\H) \simeq \K(\H)$, the algebra $\O_{\G} \otimes P$, where $P \in K(H)$ is a projection, is a $\ca$-algebra consisting of compact operators. However, it is easy to see that the spectrum of $L_w\otimes P$ contains the unit circle, a contradiction.

\vspace{.1in}
\noindent \textbf{Claim 2}. \textit{For any pair of vertices $p,q \in \vrt$, there are finitely many paths $w$ with $s(w)=p$ and $r(w)=q$.}
\vspace{.05in}

\noindent \textit{Proof of Claim 2.} Indeed, if not then we would have a $\ca$-algebra of compact operators containing an infinite projection, which is absurd.
\vspace{0.1in}

At this point, we could observe that $\G$ is a row-finite graph that satisfies condition (K) and therefore appeal to \cite[Theorem 4.9]{Raeb}, in order to prove that $\G$ is $p_0$-accessible. We find the following elementary argument illuminating, as well as making the proof self-contained.

Since $L_{p_0}$ is a full projection, there exist finitely many vertices $p_1, p_2, \dots p_n$ and paths $w_1, w_2, \dots, w_n$ and $v_1, v_2, \dots, v_n$ so that
\begin{align*}
\phantom{XXXXXX}s(w_i)&=p_i \,&r(w_i)&= q\\
s(v_i)&=p_i \, &r(v_i)&= p_{0}, &i=1,2, \dots, n,
\end{align*}
and the final spaces of $L_{w_i}L_{v_i}^{*}$, $i=1, 2, \dots, n$, span $L_{p_0}$. By Claim~2, there is no loss of generality assuming that the collection $\{ w_i\mid i=1, 2, \dots , n\}$ contains \textit{all} paths starting at one of the vertices $p_i$, $i=1, \dots,n$, and ending at $q$.

If all the vertices involved in the paths $w_1, w_2, \dots, w_n$ are accessible, then $q$ is accessible and there is nothing to prove. Otherwise, let $q_i$ be the first vertex in each path $w_i$, which is not accessible. Therefore each $q_i$ receives an edge $e_i$ which is not accessible and an edge $e_i'$ (the one from $w_i$) which is accessible. Since there exists a path $w_i'$ so that $w_i'e_i$ ends at $q$, there exists an $1\leq j_{i} \leq n$ so that $L_{w_{j_i}}^{*}L_{w_i'e_{i}}\neq 0$. Therefore either
\begin{itemize}
\item[(i)] $w_i'e_i$ is a sub-path of $w_i$, or,
\item[(ii)] $w_i$ is a sub-path of $w_i'e_i$.
\end{itemize}
However, if (ii) was valid, then $q_i$ would be accessible because $p_i =s(w_i)$ is. But this contrary to the way the $q_1, \dots,q_n$ were chosen. Hence, each $w_i'e_i$ is a sub-path of some $w_{j_i}$.

Since $w_1'e_1$ is a sub-path of $w_{j_1}$, there exists a path from $q_{j_1}$ ending at $q_1$. Since $q_{j_1}$ is not accessible, by repeating the same argument, we now obtain a path ending at $q_{j_1}$ and starting from some $q_{j_2}$, and so on. The pigeonhole principle implies that eventually one of the $q_j$'s will repeat itself, thus obtaining a cycle in $\G$. This contradicts Claim 2. Hence the arbitrary $q \in \vrt$ is accessible and so $\G$ is accessible as well.

It remains to show (ii). By way of contradiction assume that there exist a a path $w \in \Ginfty$ ending at $p_0$ and distinct edges $e_1, e_2$ satisfying $r(e_1)=r(e_2) = s(w)$. Then,
\begin{align*}
L_{p_0}\big(L_{we_1} L_{we_1}^* \big)L_{p_0}&= L_{we_1} L_{we_1}^*  \\
                                            &= L_{w}\big(L_{e_1}L_{e_1}^* \big) L_w^*\neq L_{p_0}
\end{align*}
since $L_{e_1}L_{e_1}^*$ is a proper subspace of $L_{s(w)}$.
Hence $L_{p_0} \O_{\G}L_{p_0}$ is not a corner of $\O_{\G}$, a contradiction.
 \end{proof}

\end{document}